\documentclass[11pt, letterpaper, oneside,reqno]{amsart}

\usepackage{latexsym, amsmath, amssymb, amsfonts, amscd,bm}
\usepackage{amsthm}
\usepackage{t1enc}
\usepackage[mathscr]{eucal}
\usepackage{indentfirst}
\usepackage{graphicx, pb-diagram}
\usepackage{fancyhdr}
\usepackage{fancybox}
\usepackage{enumerate}

\usepackage[all]{xy}
\usepackage[colorlinks=true]{hyperref}
\hypersetup{linkcolor=blue, citecolor=blue, filecolor=black, urlcolor=blue}
\usepackage{tikz-cd}
\usepackage{mathtools}
\usepackage{blkarray}
\usepackage{url}
\usepackage{float}
\usepackage{bm}
\usepackage[makeroom]{cancel}
\usepackage{xcolor}
\usepackage{xfrac}
\usepackage{mdframed}

\theoremstyle{plain}
\newtheorem{thm}{Theorem}[section]

\newtheorem*{theorem*}{Theorem}
\newtheorem*{prop*}{Proposition}
\newtheorem{prop}[thm]{Proposition}

\newtheorem{lemma}[thm]{Lemma}
\newtheorem{cor}[thm]{Corollary}

\newtheoremstyle{underline}
{}        
{}              
{}              
{}    
{\large}              
{:}             
{1mm}         
{{\underline{\thmname{#1}\thmnumber{ #2}}}}  

\theoremstyle{underline}
\newtheorem*{claim*}{Claim}

\theoremstyle{definition}
\newtheorem{defi}[thm]{Definition}

\theoremstyle{remark}
\newtheorem{remark}[thm]{Remark}

\newtheorem*{ack}{Acknowledgements}

\definecolor{forest}{rgb}{0,0.5,0}



\begin{document}
	
	\title[Simultaneous deformations of symplectic forms and Lagrangians]{Simultaneous deformations of symplectic forms and Lagrangian submanifolds}

	\author{Stephane Geudens}
	\address{{\scriptsize Institute of Mathematics, Polish Academy of Sciences,  ul. Sniadeckich 8, 00-656 Warsaw, Poland}}
	\email{stephane\_geudens@hotmail.com}
	
	\author{Florian Sch\"{a}tz}
	\address{{\scriptsize KfW Westarkade, Palmengartenstraße 5-9, 60325 Frankfurt am Main, Germany}}
	\email{florian.schaetz@gmail.com}
	
	\author{Alfonso G. Tortorella}
	\address{{\scriptsize DipMat, Universit\`{a} degli Studi di Salerno, Via Giovanni Paolo II n.132, 84084  Fisciano, Italy}}
	\email{atortorella@unisa.it}
	
	
	\begin{abstract}
	Given a compact symplectic manifold $(M,\omega)$ and a compact Lagrangian submanifold $L\subset(M,\omega)$, we describe small deformations of the pair $(\omega,L)$ modulo the action by isotopies. We show that the resulting moduli space can be identified with an open neighborhood of the origin in the second relative de Rham cohomology group $H^{2}(M,L)$. This implies in particular that the moduli space is smooth and finite dimensional.
	\end{abstract}
	
	\maketitle
	
	\setcounter{tocdepth}{1} 
	\tableofcontents
	
	\section*{Introduction}
	
	
	This paper represents the first part of a project which aims to describe the local moduli space of deformations of a pair $(\omega,N)$, where $\omega$ is a symplectic form on a fixed compact manifold $M$ and $N\subset(M,\omega)$ is a compact coisotropic submanifold.
	While the general coisotropic case will be addressed in a companion paper~\cite{ours}, here we focus on the special case where the submanifold is Lagrangian.
	Hence, given a compact symplectic manifold $(M,\omega)$ and a compact Lagrangian submanifold $L\subset(M,\omega)$, we aim to achieve the following.
	\vspace{0.1cm}
		\begin{quotation}
		\textbf{Main Goal:} Study the local moduli space of simultaneous deformations of the symplectic structure $\omega$ and the Lagrangian submanifold $L$, quotienting by isotopies of diffeomorphisms. More precisely, we will study pairs $(\omega',L')$ that are $(\mathcal{C}^0\times\mathcal{C}^1)$-close to $(\omega,L)$ modulo isotopies of $\mathcal{C}^1$-small diffeomorphisms.
	\end{quotation}
		\vspace{0.1cm}

	\textbf{Context.}
	The modern approach to deformation theory, which has emerged through the works of Nijenhuis--Richardson and was further developed by Schlessinger--Stasheff, Deligne, Goldman--Millson and Drinfeld among others, centers on the following principle: ``In characteristic zero, a deformation problem is controlled by a dgLa, with quasi-isomorphic dgLas giving the same deformation theory.''
	In the context of formal deformation theory and formal moduli spaces, starting with the works of Kontsevich, Hinich and Manetti, these ideas have been turned into a theorem by Lurie and Pridham establishing an equivalence between formal deformation problems and the homotopy category of dgLas.
	
	
	When studying the deformation problem of a geometric structure $\mathcal{O}$, we say that a dgLa (or more generally an $L_\infty$-algebra) $\mathfrak{g}$ controls the deformation problem when small Maurer--Cartan (MC) elements of $\mathfrak{g}$ are in bijection with small deformations of the structure $\mathcal{O}$.
	Additionally, $\mathfrak{g}$ encodes the local moduli space if this bijection intertwines the gauge equivalence of MC elements and a certain equivalence by ``isomorphisms'' of deformations. Then the underlying dg-space controls infinitesimal deformations and the Lie bracket induced in cohomology detects obstructions.
	However, even though a quasi-isomorphism $\mathfrak{g}\rightarrow\mathfrak{h}$ determines an isomorphism $\operatorname{Def}_\mathfrak{g}\rightarrow\operatorname{Def}_\mathfrak{h}$ between the associated deformation functors (with $\mathfrak{g}$ and $\mathfrak{h}$ encoding the same formal moduli space), in general the induced map between the MC moduli spaces $\operatorname{MC}(\mathfrak{g})/\text{gauge}\rightarrow\operatorname{MC}(\mathfrak{h})/\text{gauge}$ is not a bijection.
	
	
	\vspace{0.2cm}
	\textbf{Individual Deformations.}
	The individual deformation problems of a symplectic form and a Lagrangian submanifold are  controlled by abelian dgLas, i.e.~dg-spaces.
	First, the deformation problem of a compact Lagrangian submanifold $L$ in $(M,\omega)$ is controlled by its de Rham complex $(\Omega^\bullet(L),d)$.
	Indeed, by Weinstein's Lagrangian neighborhood theorem, we can identify $\mathcal{C}^1$-small deformations of $L$ with $\mathcal{C}^1$-small closed $1$-forms on $L$. Quotienting by Hamiltonian isotopies, the local moduli space of Lagrangian deformations gets parametrized by a neighborhood of $0$ in $H^1(L)$.
	Second, the deformation problem of a symplectic form $\omega$ on a given compact manifold $M$ is controlled by the shifted de Rham complex $(\Omega^\bullet(M)[1],d)$.
	Indeed, $\mathcal{C}^0$-small deformations of $\omega$ identify with $\mathcal{C}^0$-small closed $2$-forms on $M$ and Moser's stability theorem implies that, when quotienting by isotopies, the local moduli space of symplectic deformations is parametrized by a neighborhood of $0$ in $H^2(M)$.
	
	To study simultaneous deformations of $\omega$ and $L$, it is useful to look at deformations of $\omega$ from an alternative viewpoint inspired by Poisson geometry.
	First, the relation $\widetilde{\pi}=-\widetilde{\omega}^{-1}$ identifies symplectic structures $\widetilde{\omega}$ with those Poisson structures $\widetilde{\pi}$ that are non-degenerate (a $\mathcal{C}^{0}$-open condition on $\widetilde{\pi}$).
	Next, deformations of the Poisson structure $\pi=-\omega^{-1}$ are governed by the dgLa $(\Omega^\bullet(M)[1],d,[-,-]_\pi)$, where $[-,-]_\pi$ is the Koszul bracket associated with $\pi$.
	Consequently, the relation $-\widetilde\omega^{-1}=\pi-(\wedge^{2}\pi^\sharp)\eta$ identifies deformations $\widetilde{\omega}$ of $\omega$ with $\mathcal{C}^0$-small MC elements $\eta$ of $(\Omega^\bullet(M)[1],d,[-,-]_\pi)$ and it intertwines the equivalence of symplectic forms by isotopies with the gauge equivalence of $\mathcal{C}^0$-small MC elements.
	
	Now denote by $i_L\colon L\rightarrow M$ the inclusion and set $\Omega^\bullet(M,L):=\ker(i_L^\ast)$. Since $L\subset(M,\omega)$ is Lagrangian, we have that $\Omega^\bullet(M,L)$ is a sub-dgLa of $(\Omega^\bullet(M)[1],d,[-,-]_\pi)$.
	The induced dgLa-structure $(\Omega^\bullet(M,L)[1],d,[-,-]_\pi)$ satisfies the following.
	\begin{enumerate}
		\item It controls deformations of $\omega$ such that $L$ is Lagrangian, and encodes their moduli space with respect to isotopies of diffeomorphisms preserving $L$.
		\item It is homotopy abelian by a formality result due to Fiorenza-Manetti \cite{Fiorenza}.  That is, there exists an $L_\infty$-algebra isomorphism
		\begin{equation}
			\label{eq:intro:second-step}
			(\Omega^\bullet(M,L)[1],d,[-,-]_\pi)\longrightarrow(\Omega^\bullet(M,L)[1],d).
		\end{equation}
	\end{enumerate}
	
	\textbf{Strategy.}
	Our approach utilizes the dg-space $\operatorname{cone}(i_L^\ast)=(\Omega^\bullet(M)[1]\oplus\Omega^\bullet(L),d_{i_L^\ast})$, i.e.~the mapping cone of the cochain map $i_L^\ast\colon(\Omega^\bullet(M),d)\rightarrow(\Omega^\bullet(L),d)$. For the sake of explaining our strategy, we identify $M$ with $T^{*}L$ via the Lagrangian neighborhood theorem.
	We then get an abelian embedding of $\Omega^\bullet(L)[1]$ into $(\Omega^\bullet(M)[1],[-,-]_\pi)$ which is also a splitting of the following short exact sequence
	\begin{equation*}
		\label{eq:intro:first-ses}
		\begin{tikzcd}
			0\arrow[rr]&&\Omega^\bullet(M,L)\arrow[rr]&&\Omega^\bullet(M)\arrow[rr, "i_L^\ast"]&&\Omega^\bullet(L)\arrow[rr]\arrow[ll, bend left]&&0
		\end{tikzcd}.
	\end{equation*}
	Voronov's technique of higher derived brackets \cite{Voronov} yields an $L_\infty$-algebra structure $\{\ell_k\}_{k\in\mathbb{N}}$ enriching $\operatorname{cone}(i_L^\ast)$, i.e.~an $L_\infty$-algebra structure $\{\ell_k\}_{k\in\mathbb{N}}$ with $\ell_1=d_{i_L^\ast}$ which is also an $L_\infty$-algebra extension with basis $(\Omega^\bullet(M)[1],d,[-,-]_\pi)$ and fiber $(\Omega^\bullet(L),d)$:
	\begin{equation*}
		\begin{tikzcd}[column sep=0.7cm]
			0\arrow[r]&(\Omega^\bullet(L),d)\arrow[r]&(\Omega^\bullet(M)[1]\oplus\Omega^\bullet(L),\{\ell_k\}_{k\in\mathbb{N}})\arrow[r]&(\Omega^\bullet(M)[1],d,[-,-]_\pi)\arrow[r]&0
		\end{tikzcd}.
	\end{equation*}
	By the approach to simultaneous deformation problems of Fregier--Zambon~\cite{Fregier}, the $L_\infty$-algebra $(\Omega^\bullet(M)[1]\oplus\Omega^\bullet(L),\{\ell_k\}_{k\in\mathbb{N}})$ satisfies the following.
	\begin{enumerate}
		\item It controls the formal simultaneous deformations of $\omega$ and $L$, and encodes their formal moduli space with respect to isotopies of diffeomorphisms.
		\item We have a strict quasi-isomorphism of $L_\infty$-algebras given by the inclusion
		\begin{equation}
			\label{eq:intro:first-step}
			(\Omega^\bullet(M,L)[1],d,[-,-]_\pi)\hookrightarrow(\Omega^\bullet(M)[1]\oplus\Omega^\bullet(L),\{\ell_k\}_{k\in\mathbb{N}}).
		\end{equation}
	\end{enumerate}
	
	The above helps to outline a strategy for achieving the main goal of this paper, namely splitting it into two intermediate goals.
	Since the $L_\infty$-algebras appearing in Equations~\eqref{eq:intro:second-step} and~\eqref{eq:intro:first-step} are connected by $L_\infty$-algebra quasi-isomorphisms, they encode the same formal moduli space.
	In our paper we aim to establish that these equivalences extend from formal moduli spaces to local moduli spaces, leading to two intermediate goals.
	\begin{quotation}
		\textbf{Goal \#1:} Identify the local moduli space of deformations of $(\omega,L)$ with the local moduli space of deformations of $\omega$ such that $L$ remains Lagrangian,
	\end{quotation}
	\begin{quotation}
		\textbf{Goal \#2:} Identify the latter with an open neighborhood of $0$ in the second relative de Rham cohomology group $H^2(M,L)$.
	\end{quotation}
	
	\textbf{Our Results.}
	Since we aim to describe the local moduli space of simultaneous deformations of the symplectic structure $\omega$ and the Lagrangian submanifold $L$, we will consider only pairs $(\omega',L')$ lying a sufficiently small neighborhood of $(\omega,L)$ up to a slightly finer notion of equivalence by isotopies adapted to this neighborhood.
	Using the Lagrangian neighborhood theorem, Definition~\ref{def:equiv} introduces the space $\mathcal{D}_{\mathcal{U}'}(\omega,L)/\sim$ which serves as the local model for the moduli space of simultaneous deformations around the equivalence class $[(\omega,L)]$.
	Additionally, $\operatorname{Def}_L(\omega)/\sim$ denotes the space of symplectic forms on $M$ such that $L$ is Lagrangian, modulo isotopies of diffeomorphisms preserving $L$.
	Then intermediate goal \#1 above is achieved by the following (see Proposition~\ref{prop:bijection}).
	\begin{prop*}[{\bf \#1}]
		There is a well-defined map 
		$
		q\colon\mathcal{D}_{\mathcal{U}'}(\omega,L)\rightarrow\operatorname{Def}_L(\omega): (\omega',L')\mapsto\phi_{L'}^\ast\omega'
		$
		which induces a  bijection between moduli spaces $\overline{q}\colon\mathcal{D}_{\mathcal{U}'}(\omega,L)/\sim\longrightarrow\operatorname{Def}_L(\omega)/\sim$.
	\end{prop*}
	
	In the above, the map $\phi_{L'}$ is a certain diffeomorphism moving $L$ to $L'$ that is constructed explicitly out of $L'$ (see Definition~\ref{def:phi} for more details).
	Next, intermediate goal \#2 above is achieved by the following result (see Proposition~\ref{prop:H2}).
		\begin{prop*}[{\bf \#2}]
			The map $\operatorname{Def}_L(\omega)/\sim\rightarrow H^2(M,L): [\omega']\mapsto[\omega-\omega']$ identifies the moduli space of a $\mathcal{C}^0$-open around $\omega$ in $\operatorname{Def}_L(\omega)$ with an open around $0$ in $H^2(M,L)$.
	\end{prop*}
	Combining Propositions~(\#1) and (\#2) yields our main result (see Theorem~\ref{thm:main}).
	\begin{theorem*}
		The map $\mathcal{D}_{\mathcal{U}'}(\omega,L)/\sim\rightarrow H^2(M,L): [(\omega',L')]\mapsto[\omega-\phi_{L'}^\ast\omega']$ identifies the moduli space of a $\mathcal{C}^0\times\mathcal{C}^1$-open around $(\omega,L)$ in $\mathcal{D}_{\mathcal{U}'}(\omega,L)$ with an open around $0$ in $H^2(M, L)$.
	\end{theorem*}	
	\textbf{Structure of the paper.}
	The paper consists of five sections.
	In Section 1, we introduce the local model for the moduli space of simultaneous deformations.
	Section 2 reduces the simultaneous deformation problem of $(\omega,L)$ to the deformation problem of symplectic forms $\omega$ for which $L$ remains Lagrangian.
	In Section 3, we prove the main theorem stated above.
	Section 4 relates our result about the local moduli space of simultaneous deformations with the local moduli spaces of the individual deformation problems.
	Finally, Section 5 connects our findings to the result by Fiorenza and Manetti \cite{Fiorenza} on formality of Koszul brackets.
	
	\begin{ack}
	We would like to thank Marco Zambon for useful comments on a first draft of this paper. S.G. acknowledges support from the UCL Institute for Mathematical and Statistical Sciences (IMSS) and the Mathematical Institute of the Polish Academy of Sciences (IMPAN).
	A.G.T.~is a member of the National Group for Algebraic and Geometric Structures, and their Applications (GNSAGA – INdAM).
	\end{ack}

	\section{The simultaneous deformation problem}\label{sec:one}
	Assume we are given a compact symplectic manifold $(M,\omega)$ and a compact Lagrangian submanifold $L\subset(M,\omega)$. We aim to study simultaneous deformations of $\omega$ and $L$. That is, we will parametrize small deformations of the pair $(\omega,L)$ inside the space
	\begin{equation}\label{eq:def-space}
		\mathcal{D}(\omega,L):=\left\{(\omega',L'):\ \omega'\in\Omega^{2}(M)\ \text{is symplectic and}\ L'\subset(M,\omega')\ \text{is Lagrangian}\right\},
	\end{equation}
	quotienting by the equivalence relation 	
	\begin{equation}\label{eq:equiv}
		(\omega',L')\sim (\omega'',L'')\Leftrightarrow \exists\phi\in\text{Diff}_{0}(M)\ \text{such that}\ \phi^{*}\omega''=\omega'\ \text{and}\ \phi(L')=L''.
	\end{equation}
	Here $\text{Diff}_{0}(M)$ denotes the space of diffeomorphisms of $M$ that are isotopic to the identity. 
	
	\vspace{0.2cm}
	
	Rather than studying the entire moduli space $\mathcal{D}(\omega,L)/\sim$, we will consider pairs $(\omega',L')$ lying in a sufficiently small neighborhood of $(\omega,L)$, up to a slightly finer notion of equivalence that is adapted to this neighborhood. We will now make this precise. In order to restrict to Lagrangians $L'$ close to $L$, we need to recall the notion of non-linear Grassmanian.
	
	\begin{defi}
		The non-linear Grassmannian $\text{Gr}_{L}(M)$ is the Fréchet manifold consisting of all submanifolds of $M$ diffeomorphic to $L$.
	\end{defi}
	
	We will restrict to submanifolds $L'\subset M$ lying in a chart for  $\text{Gr}_{L}(M)$ around $L$. Such a chart can be obtained as follows \cite[\S\,2.1]{Grassmannian}. Using Weinstein's Lagrangian neighborhood theorem \cite{Weinstein}, we fix a symplectomorphism between a neighborhood $V$ of $L$ in $(T^{*}L,\omega_{can})$ and a neighborhood $U$ of $L$ in $(M,\omega)$, which restricts to the identity on $L$. Denote it by
	\begin{equation}\label{eq:symplecto}
	\psi:(V,\omega_{can})\rightarrow (U,\omega).
	\end{equation}
	For reasons that will become clear later, let us also fix smaller neighborhoods $V'$ and $U'$ of $L$ which correspond under $\psi$ and satisfy  $L\subset V'\subset\overline{V'}\subset V$ and $L\subset U'\subset\overline{U'}\subset U$. We also make sure that $V'$ is fiberwise convex. Now consider the set
	\[
	\Gamma_{V'}(T^{*}L)=\{\sigma\in\Gamma(T^{*}L):\ \sigma(L)\subset V'\},
	\]
	which is an open subset of the Fréchet space $\Gamma(T^{*}L)$. Let $\mathcal{U}'\subset \text{Gr}_{L}(M)$ be the neighborhood of $L$ consisting of submanifolds that are images of $\psi\circ\sigma:L\rightarrow M$ for $\sigma\in\Gamma_{V'}(T^{*}L)$. This way, we obtain a chart $\mathcal{U}'\rightarrow\Gamma_{V'}(T^{*}L)$ for $\text{Gr}_{L}(M)$ around $L$. We can now introduce a local model for the moduli space $\mathcal{D}(\omega,L)/\sim$, which will be our main object of study.
	
	\begin{defi}\label{def:equiv}
		Denote by $\Omega_{symp}(M)$ the space of symplectic forms on $M$.
		\begin{enumerate}[i)]
			\item We introduce the space
			\[
			\mathcal{D}_{\mathcal{U}'}(\omega,L):=\left\{(\omega',L')\in\Omega_{symp}(M)\times\mathcal{U}': L'\subset(M,\omega')\ \text{is Lagrangian}\right\}.
			\]	
			\item We refine the equivalence relation $\sim$ from \eqref{eq:equiv} as follows. If $(\omega',L'),(\omega'',L'')\in\mathcal{D}_{\mathcal{U}'}(\omega,L)$, we say that $(\omega',L')\sim(\omega'',L'')$ if there exists an isotopy $(\phi_t)_{t\in[0,1]}$ of $M$ such that 
			\[
			\phi_1(L')=L'',\hspace{0.5cm}(\phi_1)^{*}\omega''=\omega'\hspace{0.5cm}\text{and}\hspace{0.5cm}\phi_t(L')\in\mathcal{U}'\ \text{for all}\ t\in[0,1].
			\]
		\end{enumerate}
	\end{defi}
	
	To study the local model $\mathcal{D}_{\mathcal{U}'}(\omega,L)/\sim$ more effectively, it will be useful to provide an alternative description for it. This is the main aim of the next section.

\section{An alternative point of view}\label{sec:two}
In this section, we will prove that the local moduli space of simultaneous deformations $\mathcal{D}_{\mathcal{U}'}(\omega,L)/\sim$ can be identified with the moduli space of symplectic forms on $M$ for which $L$ is Lagrangian. This will be convenient because the latter can be parametrized more easily. 

\subsection{A map between moduli spaces}
Let us first introduce our notation for the moduli space of symplectic forms on $M$ with respect to which $L$ is Lagrangian.

	\begin{defi}
	Let $(M,\omega)$ be a compact symplectic manifold and $L\subset(M,\omega)$ a compact Lagrangian submanifold.
	We define the space
	\[
	\text{Def}_{L}(\omega):=\left\{\omega'\in\Omega_{symp}(M): L\subset(M,\omega')\ \text{is Lagrangian}\right\}.
	\]	
	For $\omega',\omega''\in\text{Def}_{L}(\omega)$, we say that $\omega'\sim\omega''$ if there is an isotopy $(\rho_t)_{t\in[0,1]}$ of $M$ such that 
	\[
	(\rho_1)^{*}\omega''=\omega'\hspace{0.5cm}\text{and}\hspace{0.5cm}\rho_t(L)=L\ \text{for all}\ t\in[0,1].
	\]
\end{defi}
	
The geometric intuition behind the correspondence between $\mathcal{D}_{\mathcal{U}'}(\omega,L)/\sim$ and $\text{Def}_{L}(\omega)/\sim$ is the following. Given  $(\omega',L')\in\mathcal{D}_{\mathcal{U}'}(\omega,L)$, one can construct a diffeomorphism isotopic to the identity which maps $L$ to $L'$. This means that $(\omega',L')$ is equivalent in $\mathcal{D}_{\mathcal{U}'}(\omega,L)$ with a pair of the form $(\omega'',L)$. To make this idea more precise, we fix some data and notation.

\begin{itemize}
	\item By the smooth Urysohn lemma, there exists a smooth function $f\in C^{\infty}(M)$ such that $f|_{\overline{U'}}\equiv 1$ and $\text{supp}(f)\subset U$. We fix such a function $f$ once and for all.
	\item Given $L'\in\mathcal{U}'$, there exists a unique section $\sigma\in\Gamma_{V'}(T^{*}L)$ such that $L'$ is the image of $\psi\circ\sigma:L\rightarrow M$. We can view the section $\sigma$ as a vertical, fiberwise constant vector field $X_{\sigma}$ on $T^{*}L$. Its flow at time $1$ takes $L$ to the submanifold $\text{Graph}(\sigma)$. 
\end{itemize}	
	
\begin{remark}\label{rem:corr}
The correspondence between sections $\sigma\in\Gamma(T^{*}L)$ and vertical, fiberwise constant vector fields $X_{\sigma}$ on $T^{*}L$ is given by the explicit formula 
\[
\iota_{X_{\sigma}}\omega_{can}=-p^{*}\sigma,
\]
where $p:T^{*}L\rightarrow L$ is the bundle projection. This is easily checked in cotangent coordinates $(x_1,\ldots,x_n,y_1,\ldots,y_n)$ on $T^{*}L$, since in such a chart we can write
\[
\omega_{can}=\sum_{i=1}^{n}dx_i\wedge dy_i,\hspace{0.5cm}\sigma=\sum_{i=1}^{n}g_idx_i\hspace{0.3cm}\text{and}\hspace{0.3cm}X_{\sigma}=\sum_{i=1}^{n}g_i\partial_{y_i}.
\]
\end{remark}

\begin{defi}\label{def:phi}
	For any $L'\in\mathcal{U}'$, we define $\phi_{L'}\in\text{Diff}_{0}(M)$ to be the flow at time $1$ of the vector field $X_{L'}:=f\psi_{*}(X_{\sigma})\in\mathfrak{X}(M)$. Note that it satisfies $\phi_{L'}(L)=L'$ and $\phi_{L}=\text{Id}$.
\end{defi}

We can now consider the map
\[
q:\mathcal{D}_{\mathcal{U}'}(\omega,L)\rightarrow\text{Def}_{L}(\omega):(\omega',L')\mapsto\phi_{L'}^{*}\omega'.
\]
It turns out that this map descends to a well-defined map at the level of moduli spaces.

\begin{lemma}\label{lem:well-defined}
	We have a well-defined map
	\[
	\overline{q}:\frac{\mathcal{D}_{\mathcal{U}'}(\omega,L)}{\sim}\rightarrow\frac{\text{Def}_{L}(\omega)}{\sim}:[(\omega',L')]\mapsto\big[\phi_{L'}^{*}\omega'\big].
	\]
\end{lemma}
\begin{proof}
Assume that $(\omega',L')$ and $(\omega'',L'')$ are equivalent in $\mathcal{D}_{\mathcal{U}'}(\omega,L)$. This means that there exists an isotopy $(\rho_t)_{t\in[0,1]}$ such that 
\[
\rho_1(L')=L'',\hspace{0.5cm}(\rho_1)^{*}\omega''=\omega'\hspace{0.5cm}\text{and}\hspace{0.5cm}\rho_t(L')\in\mathcal{U}'\ \text{for all}\ t\in[0,1].
\]
Note that $\phi_{L''}^{-1}\circ\rho_1\circ\phi_{L'}$ is a diffeomorphism preserving $L$. We will show that $\phi_{L''}^{-1}\circ\rho_1\circ\phi_{L'}$ is isotopic to the identity through an isotopy $(\sigma_t)_{t\in[0,1]}$ satisfying $\sigma_{t}(L)=L$ for all $t\in[0,1]$. This will imply that $\phi_{L'}^{*}\omega'$ and $\phi_{L''}^{*}\omega''$ are equivalent in $\text{Def}_{L}(\omega)$, because
\[
\sigma_{1}^{*}\big(\phi_{L''}^{*}\omega''\big)=\big(\phi_{L''}^{-1}\circ\rho_1\circ\phi_{L'}\big)^{*}\big(\phi_{L''}^{*}\omega''\big)=\phi_{L'}^{*}(\rho_{1}^{*}\omega'')=\phi_{L'}^{*}\omega'.
\]
In fact, it suffices to construct an isotopy $(\tau_t)_{t\in[0,1]}$ connecting $\tau_0=\text{Id}$ and $\tau_1=\phi_{L''}^{-1}\circ\rho_1\circ\phi_{L'}$ and satisfying $\tau_{t}(L)\in\mathcal{U}'$ for all $t\in[0,1]$. Indeed, we can then define a smooth family of diffeomorphisms $\phi_{\tau_{t}(L)}$ as in Def.\,\ref{def:phi} satisfying
\[
\begin{cases}
	\phi_{\tau_{0}(L)}=\phi_{L}=\text{Id},\\
	\phi_{\tau_{1}(L)}=\phi_{(\phi_{L''}^{-1}\circ\rho_1\circ\phi_{L'})(L)}=\phi_{L}=\text{Id},\\
	\phi_{\tau_{t}(L)}(L)=\tau_{t}(L).
\end{cases}
\]
If we set $\sigma_t:=\phi_{\tau_{t}(L)}^{-1}\circ \tau_t$, then we get an isotopy $(\sigma_t)_{t\in[0,1]}$ which satisfies all requirements. 

We now proceed by constructing an isotopy $(\tau_t)_{t\in[0,1]}$ as described above. Fix a smooth function $h:[0,1]\rightarrow[0,1]$ satisfying
\[
\begin{cases}
	h(0)=0,\\
	\left.h\right|_{\left]0,\frac{1}{5}\right[}:\left]0,\frac{1}{5}\right[\rightarrow h\left(\left]0,\frac{1}{5}\right[\right)\ \text{is a diffeomorphism},\\
	\left.h\right|_{[\frac{1}{5},\frac{2}{5}]}\equiv\frac{1}{3},\\
	\left.h\right|_{\left]\frac{2}{5},\frac{3}{5}\right[}:\left]\frac{2}{5},\frac{3}{5}\right[\rightarrow h\left(\left]\frac{2}{5},\frac{3}{5}\right[\right)\ \text{is a diffeomorphism},\\
	\left.h\right|_{[\frac{3}{5},\frac{4}{5}]}\equiv\frac{2}{3},\\
	\left.h\right|_{\left]\frac{4}{5},1\right[}:\left]\frac{4}{5},1\right[\rightarrow h\left(\left]\frac{4}{5},1\right[\right)\ \text{is a diffeomorphism},\\
	h(1)=1.
\end{cases}
\]
Recall from Def.\,\ref{def:phi} that $\phi_{L'}$ is the time $1$-flow of the vector field $X_{L'}$. Let us denote this flow by $\phi^{t}_{X_{L'}}$. Similarly, we have that $\phi_{L''}$ is the flow $\phi^{t}_{X_{L''}}$ at time $1$. We define $\tau_t$ by setting
\begin{align*}
	\tau_t=\begin{cases}
		\phi^{3h(t)}_{X_{L'}}\ \ \ \ &\text{if}\ 0\leq t\leq\frac{1}{5},\\
		\phi^{1}_{X_{L'}}\ \ &\text{if}\ \frac{1}{5}\leq t\leq \frac{2}{5},\\
		\rho_{(3h(t)-1)}\circ\phi^{1}_{X_{L'}} &\text{if}\ \frac{2}{5}\leq t\leq \frac{3}{5},\\
		\rho_{1}\circ\phi^{1}_{X_{L'}} &\text{if}\ \frac{3}{5}\leq t\leq \frac{4}{5},\\
		\phi^{(2-3h(t))}_{X_{L''}}\circ\rho_{1}\circ\phi^{1}_{X_{L'}}&\text{if}\ \frac{4}{5}\leq t\leq 1.
	\end{cases}
\end{align*}
Then $\tau_t$ is smooth in $t$ and it connects $\tau_0=\text{Id}$ with $\tau_1=\phi^{-1}_{X_{L''}}\circ\rho_{1}\circ\phi^{1}_{X_{L'}}=\phi_{L''}^{-1}\circ\rho_1\circ\phi_{L'}$. We now check that $\tau_t(L)\in\mathcal{U}'$ for all $t\in[0,1]$. So we have to check that $\tau_t(L)$ corresponds with a section in $\Gamma_{V'}(T^{*}L)$. As $L',L''\in\mathcal{U}'$, they correspond with sections $\alpha',\alpha''\in\Gamma_{V'}(T^{*}L)$. 
\begin{itemize}
	\item For $0\leq t\leq\frac{1}{5}$, we have that $\tau_t(L)$ corresponds with $3h(t)\alpha'\in\Gamma_{V'}(T^{*}L)$.
	\item For $\frac{1}{5}\leq t\leq\frac{2}{5}$, we have that $\tau_t(L)$ corresponds with $\alpha'\in\Gamma_{V'}(T^{*}L)$.
	\item For $\frac{2}{5}\leq t\leq\frac{3}{5}$, we have that $0\leq 3h(t)-1\leq 1$. By assumption, $\rho_{s}(L')\in\mathcal{U}'$ for all $0\leq s\leq 1$. It follows that $\tau_t(L)=\rho_{(3h(t)-1)}(L')\in\mathcal{U}'$.
	\item For $\frac{3}{5}\leq t\leq\frac{4}{5}$, we have that $\tau_t(L)=\rho_1(L')=L''$ corresponds with $\alpha''\in\Gamma_{V'}(T^{*}L)$.
	\item For $\frac{4}{5}\leq t\leq 1$, we have that $\tau_t(L)$ corresponds with the section 
	$$
	\alpha''+(2-3h(t))\alpha''=3(1-h(t))\alpha''\in\Gamma_{V'}(T^{*}L).
	$$
\end{itemize}
In the first and last bullet point, we used that $V'\subset T^{*}L$ is fiberwise convex to conclude that the scalar multiples of $\alpha'$ and $\alpha''$ still take values in $V'$. The proof is now finished.
\end{proof}

	\begin{remark}
	The key point in the proof above is that the diffeomorphism $\phi_{L''}^{-1}\circ\rho_1\circ\phi_{L'}$ leaving $L$ invariant is isotopic to $\text{Id}$ through an isotopy that leaves $L$ invariant at each stage.

	In general, if $L$ is a compact Lagrangian submanifold of a compact symplectic manifold $(M,\omega)$ and  $f\in\text{Diff}_{0}(M)$ is such that $f(L)=L$, one can not always find an isotopy connecting $f$ with the identity that fixes $L$ at each stage. For instance, let $M$ be the unit sphere in $\mathbb{R}^{3}$ with its standard symplectic form and let $L$ be the equator. Set $f:=\text{Rot}_{x,\pi}$ to be the rotation around the $x$-axis over an angle $\pi$. Then $f$ fixes $L$ and it is isotopic to the identity via the isotopy $(\text{Rot}_{x,t\pi})_{t\in[0,1]}$. However, there is no isotopy connecting $f$ with the identity and fixing $L$ at each stage. This would imply that the diffeomorphism $f|_{L}\in\text{Diff}(L)$ is isotopic to the identity, which is impossible. Indeed, since $f|_{L}$ is the reflection around the $x$-axis, it is orientation reversing hence not isotopic to $\text{Id}$. 
	The statement is true however when the diffeomorphism $f$ fixing $L$ is sufficiently $\mathcal{C}^{1}$-small, see the proof of \cite[Lemma~6]{log-symplectic}.   
\end{remark}

In what follows, we will show that the map $\overline{q}$ is a bijection between the moduli spaces $\mathcal{D}_{\mathcal{U}'}(\omega,L)/\sim$ and $\text{Def}_{L}(\omega)/\sim$. We will first study $\overline{q}$ at the infinitesimal level, showing that its formal linearization at the class $[(\omega,L)]$ is a linear isomorphism.

\subsection{The infinitesimal equivalence}
To show that the formal linearization
\begin{equation}\label{eq:linearization}
d\overline{q}:T_{[(\omega,L)]}\big(\mathcal{D}_{\mathcal{U}'}(\omega,L)/\sim\big)\rightarrow T_{[\omega]}\big(\text{Def}_{L}(\omega)/\sim\big)
\end{equation}
is an isomorphism, we have to study the infinitesimal moduli spaces appearing in \eqref{eq:linearization}. We will identify them with suitable cohomology groups governing the respective moduli problems.

\subsubsection{The infinitesimal counterpart of $\mathcal{D}_{\mathcal{U}'}(\omega,L)/\sim$}
The deformation problem of the pair $(\omega,L)$ deforms two objects while requiring that a certain compatibility condition between them remains satisfied. The complex underlying such a deformation problem is usually the mapping cone of some cochain map between the relevant deformation complexes which captures how the deformation of one object is constrained by 
the deformation of the other\footnote{A recent example of this philosophy can be found in \cite{groupoids}. There one shows that the deformation complex of a symplectic groupoid $(\mathcal{G},\omega)$ is the mapping cone of a cochain map between the deformation complex of the Lie groupoid $\mathcal{G}$ and the deformation complex of the closed multiplicative two-form $\omega$.}.

\begin{defi}
	Let $(A^{\bullet},d_{A})$ and $(B^{\bullet},d_B)$ be complexes and $\Phi:(A^{\bullet},d_A)\rightarrow(B^{\bullet},d_B)$ a cochain map. The mapping cone $\mathcal{C}(\Phi)$ is the complex $\big(A^{\bullet}\oplus B^{\bullet-1},d_{\Phi}\big)$ with differential
	\[
	d_{\Phi}(a,b)=\big(d_{A}a,\Phi(a)-d_{B}b\big),\hspace{1cm}\text{for}\ a\in A^{k}, b\in B^{k-1}.
	\]
\end{defi}

In our situation, we are given a Lagrangian submanifold $L\subset(M,\omega)$ with inclusion map $\iota_{L}:L\hookrightarrow M$. The deformation complex of the Lagrangian submanifold $L$ is $\big(\Omega^{\bullet}(L),d\big)$ while the deformation complex of the symplectic form $\omega$ is $\big(\Omega^{\bullet}(M)[1],d\big)$, as we recall in \S\,\ref{sec:relation}. The condition of being Lagrangian is captured by the cochain map $\iota_{L}^{*}:\big(\Omega^{\bullet}(M),d\big)\rightarrow\big(\Omega^{\bullet}(L),d\big)$. Hence, it is natural to expect that the mapping cone  $\big(\mathcal{C}(\iota_{L}^{*}),d_{\iota_{L}^{*}}\big)$ of the pullback $\iota_{L}^{*}$ is the deformation complex of the pair $(\omega,L)$. We will now confirm that this is indeed the case. 

To argue what is the appropriate notion of first order deformation of $(\omega,L)$, we first need a suitable definition of smooth paths in $\mathcal{D}_{\mathcal{U}'}(\omega,L)$. This is provided by the following. Recall that a submanifold $L'\in\mathcal{U}'$ corresponds with a section $\alpha\in \Gamma_{V'}(T^{*}L)$.

\begin{defi}
A path $(\omega_t,L_t)$ in $\mathcal{D}_{\mathcal{U}'}(\omega,L)$ is smooth if $\omega_t$ is a smooth path in $\Omega^{2}(M)$ and the path $\alpha_t$ in $\Gamma_{V'}(T^{*}L)$ corresponding with $L_t$ is smooth.
\end{defi}

\begin{lemma}\label{lem:first-order}
Let $(\omega_t,L_t)$ be a smooth path in $\mathcal{D}_{\mathcal{U}'}(\omega,L)$ starting at $(\omega,L)$. Denote by $\alpha_t$ the path in $\Gamma_{V'}(T^{*}L)$ corresponding with $L_t$. Then the following hold:
\begin{enumerate}
	\item The infinitesimal deformation $(\dot{\omega}_0,\dot{\alpha}_0)$ is a cocycle in  $\big(\mathcal{C}(\iota_{L}^{*}),d_{\iota_{L}^{*}}\big)$.
	\item If the path $(\omega_t,L_t)$ is generated by an isotopy $\phi_t$, meaning that
	\[
	\phi_t(L)=L_t\hspace{0.2cm}\text{and}\hspace{0.2cm}\phi_t^{*}\omega_t=\omega,
	\]
	then the infinitesimal deformation $(\dot{\omega}_0,\dot{\alpha}_0)$ is a coboundary in $\big(\mathcal{C}(\iota_{L}^{*}),d_{\iota_{L}^{*}}\big)$.
\end{enumerate}
\end{lemma}
\begin{proof}
For item $(1)$, it is clear that $\dot{\omega}_0$ is closed. It remains to show that
\begin{equation}\label{eq:closedness}
\iota_{L}^{*}\dot{\omega}_0=d\dot{\alpha}_0
\end{equation}
Because the symplectomorphism $\psi:(V,\omega_{can})\rightarrow (U,\omega)$ that we fixed in \eqref{eq:symplecto} restricts to the identity on $L$, we may assume that $\omega_t$ is a path of symplectic forms on $V\subset T^{*}L$ starting at $\omega_{can}$ such that the graph of $\alpha_t\in\Gamma_{V'}(T^{*}L)$ is Lagrangian with respect to $\omega_t$. This means that $\alpha_{t}^{*}\omega_t=0$. Differentiating this equality at time $t=0$, we obtain 
\[
0=\left.\frac{d}{dt}\right|_{t=0}\alpha_{t}^{*}\omega_{can}+\iota_{L}^{*}\dot{\omega}_0=-d\left(\left.\frac{d}{dt}\right|_{t=0}\alpha_{t}^{*}\theta_{taut}\right)+\iota_{L}^{*}\dot{\omega}_0=-d\dot{\alpha}_0+\iota_{L}^{*}\dot{\omega}_0.
\]
Here we used that the tautological one-form $\theta_{taut}\in\Omega^{1}(T^{*}L)$ satisfies $\alpha^{*}\theta_{taut}=\alpha$ for all $\alpha\in\Gamma(T^{*}L)$, see \cite[Prop.\,3.1.18]{McDuff}. This shows that \eqref{eq:closedness} holds, which proves item $(1)$.

For item $2)$, let $Y_t$ be the time dependent vector field defined by the isotopy $\phi_t$. We claim that the infinitesimal deformation $(\dot{\omega}_0,\dot{\alpha}_0)$ satisfies
\begin{equation}\label{eq:exact}
(\dot{\omega}_0,\dot{\alpha}_0)=-\big(d\iota_{Y_0}\omega,\iota_{L}^{*}(\iota_{Y_0}\omega)\big)=d_{\iota_{L}^{*}}(-\iota_{Y_0}\omega,0).
\end{equation}
First, differentiating the equality $\phi_t^{*}\omega_t=\omega$ at time $t=0$ gives
\[
0=\left.\frac{d}{dt}\right|_{t=0}\phi_t^{*}\omega_t=\pounds_{Y_0}\omega+\dot{\omega}_0=d\iota_{Y_0}\omega+\dot{\omega}_0,
\]
which shows that the first components in \eqref{eq:exact} are equal. To show that the second components in \eqref{eq:exact} are also equal, we will again use that the symplectomorphism $\psi:(V,\omega_{can})\rightarrow (U,\omega)$ restricts to the identity on $L$. Denoting by $Z_t:=(\psi^{-1})_{*}Y_t$, it then suffices to prove that
\begin{equation}\label{eq:toprove}
\dot{\alpha}_0=-\iota_{L}^{*}\big(\iota_{Z_0}\omega_{can}\big).
\end{equation}
Note that the isotopy $\psi^{-1}\circ\phi_t\circ\psi$ integrating the time-dependent vector field $Z_t$ satisfies $\big(\psi^{-1}\circ\phi_t\circ\psi\big)(L)=\text{graph}(\alpha_t)$. By \cite[Lemma\,3.13]{equivalences}, this implies that
\begin{equation}\label{eq:dot}
\dot{\alpha}_0=P\big(Z_0|_{L}\big),
\end{equation}
where $P$ is the vertical projection in the decomposition $T(T^{*}L)|_{L}=TL\oplus T^{*}L$. Unravelling this equality using the explicit correspondence between sections of $T^{*}L$ and vertical, fiberwise constant vector fields on $T^{*}L$ from Rem.\,\ref{rem:corr}, the expression \eqref{eq:dot} in fact states that
\[
\dot{\alpha}_0=-\iota_{L}^{*}\left(\iota_{P\big(Z_0|_{L}\big)}\omega_{can}\right)=-\iota_{L}^{*}\big(\iota_{Z_0}\omega_{can}\big).
\]
In the last equality, we used that $\iota_{L}^{*}\omega_{can}=0$. This shows that the equation \eqref{eq:toprove} holds.
\end{proof}

The following result is a converse to Lemma\,\ref{lem:first-order}. It shows in particular that the simultaneous deformation problem of $(\omega,L)$ is smoothly unobstructed.

\begin{lemma}\label{lem:unobstructed}
Let $(\eta,\beta)\in\Omega^{2}(M)\times\Omega^{1}(L)$ be a cocycle in $\big(\mathcal{C}(\iota_{L}^{*}),d_{\iota_{L}^{*}}\big)$.
\begin{enumerate}
\item There exists a smooth path $(\omega_t,L_t)$ in $\mathcal{D}_{\mathcal{U}'}(\omega,L)$ starting at $(\omega,L)$ whose associated infinitesimal deformation $(\dot{\omega}_0,\dot{\alpha}_0)$ is exactly $(\eta,\beta)$.
\item If $(\eta,\beta)$ is a coboundary in $\big(\mathcal{C}(\iota_{L}^{*}),d_{\iota_{L}^{*}}\big)$, then one can find such a path $(\omega_t,L_t)$ that is generated by an isotopy $\phi_t$.
\end{enumerate}
\end{lemma}

Before proving Lemma\,\ref{lem:unobstructed}, we make some useful observations.
\begin{itemize}
	\item We saw in equation \eqref{eq:exact} that the first order deformation coming from a path $(\omega_t,L_t)$ generated by an isotopy is a coboundary lying in $d_{\iota_{L}^{*}}\big(\Omega^{1}(M)\oplus\{0\}\big)$. In fact, every two-coboundary is of this form since
	\begin{equation}\label{eq:coboundaries}
	d_{\iota_{L}^{*}}\big(\Omega^{k}(M)\oplus\{0\}\big)=d_{\iota_{L}^{*}}\big(\Omega^{k}(M)\oplus \Omega^{k-1}(L)\big).
	\end{equation}
	Indeed, given $(\eta,\beta)\in\Omega^{k}(M)\oplus \Omega^{k-1}(L)$ one can always find $\zeta\in\Omega^{k}(M)$ such that
	\[
	d_{\iota_{L}^{*}}(\zeta,0)=(d\zeta,\iota_{L}^{*}\zeta)=(d\eta,\iota_{L}^{*}\eta-d\beta)=d_{\iota_{L}^{*}}(\eta,\beta).
	\]
	For instance, let $\widetilde{\beta}\in \Omega^{k-1}(M)$ be any extension of $\beta\in \Omega^{k-1}(L)$ and set $\zeta:=\eta-d\widetilde{\beta}$.
	\item Let $(\eta,\beta)\in\Omega^{2}(M)\times\Omega^{1}(L)$ be a two-cocycle in $\big(\mathcal{C}(\iota_{L}^{*}),d_{\iota_{L}^{*}}\big)$, that is
	\[
	\begin{cases}
		d\eta=0\\
		\iota_{L}^{*}\eta=d\beta
	\end{cases}.
	\]
	If $\widetilde{\beta}\in\Omega^{1}(M)$ is any extension of $\beta$, then we have
	\begin{equation}\label{eq:dec-inf}
	(\eta,\beta)=(\eta-d\widetilde{\beta},0)+(d\widetilde{\beta},\beta)=(\eta-d\widetilde{\beta},0)+d_{\iota_{L}^{*}}(\widetilde{\beta},0).
	\end{equation}
	Hence, the cocycle $(\eta,\beta)$ is cohomologous with the cocycle $(\eta-d\widetilde{\beta},0)$ whose second component is zero. This indicates how we should construct a path of deformations $(\omega_t,L_t)$ prolonging the first order deformation $(\eta,\beta)$. First take a path of the form $(\omega'_t,L)$ prolonging $(\eta-d\widetilde{\beta},0)$, then correct this path by applying a suitable isotopy.
\end{itemize}

\begin{proof}[Proof of Lemma\,\ref{lem:unobstructed}]
For item $(1)$, let $(\eta,\beta)\in\Omega^{2}(M)\times\Omega^{1}(L)$ be a two-cocycle in $(\mathcal{C}(\iota_{L}^{*}),d_{\iota_{L}^{*}})$ and fix an extension $\widetilde{\beta}\in\Omega^{1}(M)$ of $\beta$. We will follow the strategy outlined in the second bullet point above. First, take $\epsilon>0$ such that for all $0\leq t\leq\epsilon$, the two-form $\omega+t(\eta-d\widetilde{\beta})$ is still symplectic. Note that $L$ is Lagrangian with respect to  $\omega+t(\eta-d\widetilde{\beta})$ for all $0\leq t\leq\epsilon$. Hence we get a smooth path 
\[
\big(\omega+t(\eta-d\widetilde{\beta}),L\big)_{0\leq t\leq\epsilon}
\] 
in $\mathcal{D}_{\mathcal{U}'}(\omega,L)$, which prolongs the first order deformation $(\eta-d\widetilde{\beta},0)$. Second, let $\phi_t$ be the flow of the vector field $Y\in\mathfrak{X}(M)$ determined by $\iota_{Y}\omega=-\widetilde{\beta}$. Shrinking $\epsilon$ if needed, we can make sure that the path $(\phi_t(L))_{0\leq t\leq\epsilon}$ stays inside the neighbourhood $\mathcal{U}'\subset\text{Gr}_{L}(M)$ of $L$. We then get a smooth path
\[
\Big((\phi_{t}^{-1})^{*}\big(\omega+t(\eta-d\widetilde{\beta})\big),\phi_t(L)\Big)_{0\leq t\leq\epsilon}
\] 
in $\mathcal{D}_{\mathcal{U}'}(\omega,L)$, which prolongs the first order deformation $(\eta,\beta)$. Indeed, on one hand 
\[
\left.\frac{d}{dt}\right|_{t=0}(\phi_{t}^{-1})^{*}\big(\omega+t(\eta-d\widetilde{\beta})\big)=-\pounds_{Y}\omega+\eta-d\widetilde{\beta}=-d\iota_{Y}\omega+\eta+d\iota_{Y}\omega=\eta.
\]
On the other hand, let $(\alpha_t)_{0\leq t\leq \epsilon}$ be the path in $\Gamma_{V'}(T^{*}L)$ corresponding with the path $(\phi_t(L))_{0\leq t\leq\epsilon}$ in $\mathcal{U}'$. Since $\iota_{Y}\omega=-\widetilde{\beta}$,  checking that $\dot{\alpha}_0=\beta$ amounts to showing that
\[
\dot{\alpha}_{0}=-\iota_{L}^{*}(\iota_{Y}\omega).
\]
This is done exactly as in the proof of Lemma\,\ref{lem:first-order}\,$(2)$. This finishes the proof of item\,$(1)$.

For item\,$(2)$, recall from \eqref{eq:coboundaries} that every two-coboundary in $\big(\mathcal{C}(\iota_{L}^{*}),d_{\iota_{L}^{*}}\big)$ is of the form
\[
d_{\iota_{L}^{*}}(\zeta,0)=(d\zeta,\iota_{L}^{*}\zeta),\hspace{1cm}\zeta\in\Omega^{1}(M).
\]
Let $\phi_t$ be the flow of the vector field $Y\in\mathfrak{X}(M)$ determined by $\iota_{Y}\omega=-\zeta$. Take $\epsilon>0$ small enough such that the path $(\phi_t(L))_{0\leq t\leq\epsilon}$ stays inside the neighbourhood $\mathcal{U}'\subset\text{Gr}_{L}(M)$ of $L$. We then get a smooth path
\[
\big((\phi_{t}^{-1})^{*}\omega,\phi_t(L)\big)_{0\leq t\leq\epsilon}
\]
in $\mathcal{D}_{\mathcal{U}'}(\omega,L)$. The proof of Lemma\,\ref{lem:first-order}~ $(2)$ shows that the first order deformation of $(\omega,L)$ coming from this path is exactly given by
\[
d_{\iota_{L}^{*}}(-\iota_{Y}\omega,0)=d_{\iota_{L}^{*}}(\zeta,0). \qedhere
\]
\end{proof}

\begin{cor}\label{cor:coho1}
The formal tangent space $T_{[(\omega,L)]}\big(\mathcal{D}_{\mathcal{U}'}(\omega,L)/\sim\big)$ is given by $H^{2}\big(\mathcal{C}(\iota_{L}^{*})\big)$.
\end{cor}

\begin{remark}
This shows in particular that the formal tangent space $T_{[(\omega,L)]}\big(\mathcal{D}_{\mathcal{U}'}(\omega,L)/\sim\big)$ is finite dimensional. Indeed, it is well-known that the mapping cone 	$\big(\mathcal{C}(\iota_{L}^{*}),d_{\iota_{L}^{*}}\big)$ fits in a short exact sequence of complexes
\[
\begin{tikzcd}
	0\arrow[r]&\big(\Omega^{\bullet-1}(L),-d\big)\arrow{r}{\beta\mapsto(0,\beta)}&\big(\mathcal{C}(\iota_{L}^{*})^{\bullet},d_{\iota_{L}^{*}}\big)\arrow{r}{(\eta,\beta)\mapsto \eta}&\big(\Omega^{\bullet}(M),d)\arrow[r]&0.
\end{tikzcd}
\]
We get a long exact sequence in cohomology
\[
\cdots\longrightarrow H^{k-1}(L)\longrightarrow H^{k}\big(\mathcal{C}(\iota_{L}^{*})\big)\longrightarrow H^{k}(M)\longrightarrow\cdots
\]
Since $L$ and $M$ are compact, we know that $H^{k-1}(L)$ and $H^{k}(M)$ are finite dimensional. The long exact sequence then implies that $H^{k}\big(\mathcal{C}(\iota_{L}^{*})\big)$ is finite dimensional as well.
\end{remark}

\subsubsection{The infinitesimal counterpart of $\text{Def}_{L}(\omega)/\sim$}

It is not hard to describe the formal tangent space $T_{[\omega]}\big(\text{Def}_{L}(\omega)/\sim\big)$. The relevant cochain complex is now the relative de Rham complex $\big(\Omega^{\bullet}(M,L),d\big)$, which is defined as follows. Given the inclusion $\iota_L:L\hookrightarrow M$, we set
\[
\Omega^{k}(M,L)=\big\{\alpha\in\Omega^{k}(M):i_{L}^{*}\alpha=0\big\}.
\]
We call the cohomology groups $H^{\bullet}(M,L)$ arising from $\big(\Omega^{\bullet}(M,L),d\big)$ the relative de Rham cohomology groups of $M$ with respect to $L$. The following result is straightforward.

\begin{lemma}
$(1)$ Let $\omega_t$ be a smooth path in $\text{Def}_{L}(\omega)$ starting at $\omega$. The corresponding infinitesimal deformation $\dot{\omega}_{0}$ is a closed element of $\Omega^{2}(M,L)$. Conversely, if $\eta\in\Omega^{2}(M,L)$ is closed then there exists a smooth path $\omega_t$ in $\text{Def}_{L}(\omega)$ starting at $\omega$ such that $\dot{\omega}_{0}=\eta$.

$(2)$ If a smooth path $\omega_t$ in $\text{Def}_{L}(\omega)$ starting at $\omega$ is generated by an isotopy preserving $L$, then $\dot{\omega}_{0}\in d\Omega^{1}(M,L)$. Conversely, if $\eta\in d\Omega^{1}(M,L)$, then there is a smooth path $\omega_t$ in $\text{Def}_{L}(\omega)$ starting at $\omega$ which is generated by an isotopy preserving $L$ and satisfies $\dot{\omega}_{0}=\eta$.
\end{lemma}
\begin{proof}
$(1)$ If $\omega_t$ is a smooth path in $\text{Def}_{L}(\omega)$, then $\omega_t$ is closed and belongs to $\Omega^{2}(M,L)$ for every $t$. Hence, the same holds for $\dot{\omega}_{0}$. Next, if $\eta\in\Omega^{2}(M,L)$ is closed then $\omega_t:=\omega+t\eta$ is symplectic for small enough $t$ and moreover $L$ is Lagrangian with respect to $\omega_t$. It follows that $\omega_t$ is a smooth path in $\text{Def}_{L}(\omega)$ starting at $\omega$ such that $\dot{\omega}_{0}=\eta$.

$(2)$ By assumption, we have that $\phi_{t}^{*}\omega_t=\omega$ where  $\phi_t$ is an isotopy whose time dependent vector field $Y_t$ is tangent to $L$. Differentiating at time $t=0$, we get
\[
\dot{\omega}_{0}=-d\iota_{Y_0}\omega.
\]
Here $\iota_{Y_0}\omega\in\Omega^{1}(M,L)$ since $\omega\in\Omega^{2}(M,L)$ and $Y_0$ is tangent to $L$. Hence $\dot{\omega}_{0}\in d\Omega^{1}(M,L)$. Conversely, take a coboundary $d\gamma$ with $\gamma\in\Omega^{1}(M,L)$. Let $\phi_t$ be the flow of the vector field $Y\in\mathfrak{X}(M)$ determined by $\iota_{Y}\omega=-\gamma$. The fact that $\gamma\in\Omega^{1}(M,L)$ implies that $Y$ is tangent to $L$ because 
\[
Y|_{L}\in\Gamma(TL^{\omega})=\Gamma(TL).
\]
So the isotopy $\phi_t$ preserves $L$. The path $\omega_t:=(\phi_{t}^{-1})^{*}\omega$ generated by $\phi_t$ satisfies
\[
\dot{\omega}_{0}=-d\iota_{Y}\omega=d\gamma. \qedhere
\]
\end{proof}

\begin{cor}\label{cor:coho2}
	The formal tangent space $T_{[\omega]}\big(\text{Def}_{L}(\omega)/\sim\big)$ is given by $H^{2}(M,L)$.
\end{cor}

\begin{remark}
This shows in particular that the formal tangent space $T_{[\omega]}\big(\text{Def}_{L}(\omega)/\sim\big)$ is finite dimensional. Indeed, the relative de Rham complex $\big(\Omega^{\bullet}(M,L),d\big)$ fits in a short exact sequence of complexes
\[
0\longrightarrow \big(\Omega^{\bullet}(M,L),d\big)\longrightarrow \big(\Omega^{\bullet}(M),d\big)\overset{i_{L}^{*}}{\longrightarrow}\big(\Omega^{\bullet}(L),d\big)\longrightarrow 0,
\]
which induces a long exact sequence in cohomology
\begin{equation}\label{eq:long-ex}
	\cdots \longrightarrow H^{k-1}(L)\longrightarrow H^{k}(M,L)\longrightarrow H^{k}(M)\longrightarrow \cdots.
\end{equation}
Since $L$ and $M$ are compact, we know that $H^{k-1}(L)$ and $H^{k}(M)$ are finite dimensional. The long exact sequence then implies that $H^{k}(M,L)$ is finite dimensional as well. 
\end{remark}

\subsubsection{A linear isomorphism}
We will now prove that the formal linearization 
\[
d\overline{q}:T_{[(\omega,L)]}\big(\mathcal{D}_{\mathcal{U}'}(\omega,L)/\sim\big)\rightarrow T_{[\omega]}\big(\text{Def}_{L}(\omega)/\sim\big)
\]
is a linear isomorphism. We will proceed by first recalling a canonical isomorphism between the cohomology groups $H^{2}\big(\mathcal{C}(\iota_{L}^{*})\big)$ and $H^{2}(M,L)$. We will then show that this isomorphism is realized by the map $d\overline{q}$ under the identifications established in Cor.\,\ref{cor:coho1} and Cor.\,\ref{cor:coho2}.

\vspace{0.2cm}

It is well-known that for a surjective cochain map $\Phi:(A^{\bullet},d_A)\rightarrow(B^{\bullet},d_B)$, the natural inclusion $\big(\ker\Phi,d_A\big)\hookrightarrow\big(\mathcal{C}(\Phi),d_{\Phi}\big)$ is a quasi-isomorphism \cite[Lemma\,5.2.4]{Manetti}. Applying this fact to the pullback  $\iota_{L}^{*}:\big(\Omega^{\bullet}(M),d\big)\rightarrow\big(\Omega^{\bullet}(L),d\big)$ gives an isomorphism
\begin{equation*}
I:H^{k}(M,L)\overset{\sim}{\longrightarrow}H^{k}\big(\mathcal{C}(\iota_{L}^{*})\big):[\eta]\mapsto [(\eta,0)].
\end{equation*}
We will need the inverse of this isomorphism, which was already hinted at in equation \eqref{eq:dec-inf}.

\begin{lemma}
The inverse of the isomorphism $I$ is given by
\[
J:H^{k}\big(\mathcal{C}(\iota_{L}^{*})\big)\rightarrow H^{k}(M,L):[(\eta,\beta)]\mapsto[\eta-d\widetilde{\beta}].
\]
Here $\widetilde{\beta}\in\Omega^{k-1}(M)$ is any extension of $\beta\in\Omega^{k-1}(L)$.
\end{lemma}
\begin{proof}
We first show that $J$ is well-defined. If $(\eta,\beta)$ is a $k$-cocycle in $\big(\mathcal{C}(\iota_{L}^{*}),d_{\iota_{L}^{*}}\big)$, then
\[
\begin{cases}
	d\eta=0\\
	\iota_{L}^{*}\eta=d\beta
\end{cases}.
\]
Hence for any extension $\widetilde{\beta}$ of $\beta$, the form $\eta-d\widetilde{\beta}$ is closed and its pullback to $L$ vanishes. So it defines a class in $H^{k}(M,L)$. This class does not depend on the choice of extension $\widetilde{\beta}$. Indeed, if $\widetilde{\beta}_1$ and $\widetilde{\beta}_2$ are two extensions of $\beta$ then
\[
(\eta-d\widetilde{\beta}_1)-(\eta-d\widetilde{\beta}_2)=d(\widetilde{\beta}_2-\widetilde{\beta}_1),
\]
where $\widetilde{\beta}_2-\widetilde{\beta}_1\in\Omega^{k-1}(M,L)$. Hence $\eta-d\widetilde{\beta}_1$ and $\eta-d\widetilde{\beta}_2$ define the same class in $H^{k}(M,L)$.
So we get a well-defined map
\begin{equation}\label{eq:map}
\mathcal{C}^{k}(\iota_{L}^{*})_{closed}\rightarrow H^{k}(M,L):(\eta,\beta)\mapsto[\eta-d\widetilde{\beta}].
\end{equation}
To check that it descends to $H^{k}\big(\mathcal{C}(\iota_{L}^{*})\big)$, we show that $k$-coboundaries are mapped to zero. By equation  \eqref{eq:coboundaries}, a $k$-coboundary is necessarily of the form
\[
d_{\iota_{L}^{*}}(\zeta,0)=(d\zeta,\iota_{L}^{*}\zeta)
\]
for some $\zeta\in\Omega^{k-1}(M)$. An extension of $\iota_{L}^{*}\zeta\in\Omega^{k-1}(L)$ is simply given by $\zeta$, hence it is clear that $(d\zeta,\iota_{L}^{*}\zeta)$ gets mapped to zero by the map \eqref{eq:map}. It follows that $J$ is well-defined.

It remains to check that $I$ and $J$ are inverses. It is clear that $J\circ I$ is the identity map. On the other hand, the composition $I\circ J$ is given by
\[
I\circ J:H^{k}\big(\mathcal{C}(\iota_{L}^{*})\big)\rightarrow H^{k}\big(\mathcal{C}(\iota_{L}^{*})\big):[(\eta,\beta)]\mapsto[(\eta-d\widetilde{\beta},0)].
\]
We already remarked in \eqref{eq:dec-inf} that a cocycle $(\eta,\beta)$ is cohomologous with $(\eta-d\widetilde{\beta},0)$ since
\[
(\eta,\beta)=(\eta-d\widetilde{\beta},0)+(d\widetilde{\beta},\beta)=(\eta-d\widetilde{\beta},0)+d_{\iota_{L}^{*}}(\widetilde{\beta},0).
\]
This shows that $I\circ J$ is also the identity map. Hence $I$ and $J$ are inverses of each other.
\end{proof}

The isomorphism $J:H^{2}\big(\mathcal{C}(\iota_{L}^{*})\big)\rightarrow H^{2}(M,L)$ is in fact realized by the formal linearization of $\overline{q}$ at the class $[(\omega,L)]\in\mathcal{D}_{\mathcal{U}'}(\omega,L)/\sim$. In particular, the latter is also an isomorphism.

\begin{prop}
We have a commutative diagram
\[
\begin{tikzcd}
	&T_{[(\omega,L)]}\big(\mathcal{D}_{\mathcal{U}'}(\omega,L)/\sim\big)\arrow[r,"d\overline{q}"]\arrow[d,"\simeq",swap] & T_{[\omega]}\big(\text{Def}_{L}(\omega)/\sim\big)\arrow[d,"\simeq"]\\
	&H^{2}\big(\mathcal{C}(\iota_{L}^{*})\big)\arrow[r,"J"] &H^{2}(M,L) 
\end{tikzcd}.
\]
\end{prop}
\begin{proof}
Take a smooth path $(\omega_t,L_t)$ in $\mathcal{D}_{\mathcal{U}'}(\omega,L)$ starting at $(\omega,L)$. Denote by $\alpha_t$ the path in $\Gamma_{V'}(T^{*}L)$ corresponding with $L_t$. We need to prove the following equality in $H^{2}(M,L)$:
\begin{equation}\label{eq:prove}
\left[\left.\frac{d}{dt}\right|_{t=0}\phi_{L_t}^{*}\omega_t\right]=\left[\dot{\omega}_0-d\widetilde{\dot{\alpha}_0}\right]
\end{equation}
Denote by $Y_t$ the time-dependent vector field of the isotopy $\phi_{L_t}$. We then have
\[
\left.\frac{d}{dt}\right|_{t=0}\phi_{L_t}^{*}\omega_t=d\iota_{Y_0}\omega+\dot{\omega}_0.
\]
Hence, the equality \eqref{eq:prove} follows if we show that $\iota_{Y_0}\omega+\widetilde{\dot{\alpha}_0}\in\Omega^{1}(M,L)$. That is to say,
\[
\dot{\alpha}_0=-\iota_{L}^{*}(\iota_{Y_0}\omega).
\]
To prove this equality, we use that the symplectomorphism $\psi:(V,\omega_{can})\rightarrow(U,\omega)$ restricts to the identity on $L$. Denoting by $Z_0:=(\psi^{-1})_{*}Y_0$, it then suffices to prove that
\begin{equation}\label{eq:sufficient}
\dot{\alpha}_0=-\iota_{L}^{*}(\iota_{Z_0}\omega_{can}).
\end{equation}

We now make the vector field $Z_0$ more explicit. Recall that $\phi_{L_t}$ is the flow at time $1$ of the vector field $X_{L_t}:=f\psi_{*}(X_{\alpha_t})\in\mathfrak{X}(M)$. It will be beneficial to rephrase this as 
\[
\phi_{L_t}=\exp(X_{L_t}),
\]
where $\exp:\mathfrak{X}(M)\rightarrow\text{Diff}(M)$ denotes the exponential map of the infinite dimensional Lie group $\text{Diff}(M)$. Since the derivative $d_{0}\exp$ is the identity map, it follows that
\[
Y_0=\left.\frac{d}{dt}\right|_{t=0}\phi_{L_t}=\left.\frac{d}{dt}\right|_{t=0}\exp(X_{L_t})=(d_{0}\exp)\left(\left.\frac{d}{dt}\right|_{t=0}X_{L_t}\right)=f\psi_{*}(X_{\dot{\alpha}_0}).
\]
This implies that the vector field $Z_0=(\psi^{-1})_{*}Y_0$ is given by $Z_0=\psi^{*}(f)X_{\dot{\alpha}_0}$. The equality \eqref{eq:sufficient} now follows immediately. Since $\psi^{*}(f)$ is identically $1$ along $L$, we get
\[
-\iota_{L}^{*}(\iota_{Z_0}\omega_{can})=-\iota_{L}^{*}(\iota_{X_{\dot{\alpha}_0}}\omega_{can})=\dot{\alpha}_0.
\]
In the last equality, we used again the correspondence from Rem.\,\ref{rem:corr} between vertical, fiberwise constant vector fields on $T^{*}L$ and sections of $T^{*}L$. This finishes the proof.
\end{proof}

\subsection{The geometric equivalence}
The fact that the formal linearization of $\overline{q}$ is an isomorphism suggests that $\overline{q}$ should be a local equivalence between the moduli spaces $\mathcal{D}_{\mathcal{U}'}(\omega,L)/\sim$ and $\text{Def}_{L}(\omega)/\sim$. We now confirm that this is indeed the case.
	
\begin{prop}\label{prop:bijection}
We have a bijection
\[
\overline{q}:\frac{\mathcal{D}_{\mathcal{U}'}(\omega,L)}{\sim}\rightarrow\frac{\text{Def}_{L}(\omega)}{\sim}:[(\omega',L')]\mapsto\big[\phi_{L'}^{*}\omega'\big].
\]
\end{prop}	
\begin{proof}
Surjectivity is clear, since $[\omega']\in\text{Def}_{L}(\omega)/\sim$ is the image of $[(\omega',L)]\in\mathcal{D}_{\mathcal{U}'}(\omega,L)/\sim$.
For injectivity, assume that $(\omega',L'),(\omega'',L'')\in\mathcal{D}_{\mathcal{U}'}(\omega,L)$ are such that $\phi_{L'}^{*}\omega'$ and $\phi_{L''}^{*}\omega''$ are equivalent in $\text{Def}_{L}(\omega)$. Then there exists an isotopy $(\psi_t)_{t\in[0,1]}$ such that 
\[
(\psi_1)^{*}\big(\phi_{L''}^{*}\omega''\big)=\phi_{L'}^{*}\omega'\hspace{0.5cm}\text{and}\hspace{0.5cm}\psi_t(L)=L\ \text{for all}\ t\in[0,1].
\]
We will construct an isotopy $(\rho_t)_{t\in[0,1]}$ connecting $\rho_0=\text{Id}$ and $\rho_1=\phi_{L''}\circ\psi_1\circ\phi_{L'}^{-1}$ such that additionally $\rho_t(L')\in\mathcal{U}'$ for all $t\in[0,1]$. This will imply that $(\omega',L')$ and $(\omega'',L'')$ are equivalent in $\mathcal{D}_{\mathcal{U}'}(\omega,L)$, because
\[
\rho_1(L')=\big(\phi_{L''}\circ\psi_1\circ\phi_{L'}^{-1}\big)(L')=L''
\]
and
\[
(\rho_1)^{*}\omega''=\big(\phi_{L'}^{-1}\big)^{*}\left((\psi_1)^{*}\big(\phi_{L''}^{*}\omega''\big)\right)=\big(\phi_{L'}^{-1}\big)^{*}(\phi_{L'}^{*}\omega')=\omega'.
\]
We can proceed as in the proof of Lemma\,\ref{lem:well-defined}. Consider again the same smooth function $h:[0,1]\rightarrow[0,1]$ satisfying
\[
\begin{cases}
	h(0)=0,\\
	\left.h\right|_{\left]0,\frac{1}{5}\right[}:\left]0,\frac{1}{5}\right[\rightarrow h\left(\left]0,\frac{1}{5}\right[\right)\ \text{is a diffeomorphism},\\
	\left.h\right|_{[\frac{1}{5},\frac{2}{5}]}\equiv\frac{1}{3},\\
	\left.h\right|_{\left]\frac{2}{5},\frac{3}{5}\right[}:\left]\frac{2}{5},\frac{3}{5}\right[\rightarrow h\left(\left]\frac{2}{5},\frac{3}{5}\right[\right)\ \text{is a diffeomorphism},\\
	\left.h\right|_{[\frac{3}{5},\frac{4}{5}]}\equiv\frac{2}{3},\\
	\left.h\right|_{\left]\frac{4}{5},1\right[}:\left]\frac{4}{5},1\right[\rightarrow h\left(\left]\frac{4}{5},1\right[\right)\ \text{is a diffeomorphism},\\
	h(1)=1.
\end{cases}
\]
Recall from Def.\,\ref{def:phi} that $\phi_{L'}$ is the time $1$-flow of the vector field $X_{L'}$. Let us denote this flow by $\phi^{t}_{X_{L'}}$. Similarly, we have that $\phi_{L''}$ is the flow $\phi^{t}_{X_{L''}}$ at time $1$. We now define
\begin{align*}
	\rho_t=\begin{cases}
		\phi^{-3h(t)}_{X_{L'}}\ \ \ \ &\text{if}\ 0\leq t\leq\frac{1}{5},\\
		\phi^{-1}_{X_{L'}}\ \ &\text{if}\ \frac{1}{5}\leq t\leq \frac{2}{5},\\
		\psi_{(3h(t)-1)}\circ\phi^{-1}_{X_{L'}} &\text{if}\ \frac{2}{5}\leq t\leq \frac{3}{5},\\
		\psi_{1}\circ\phi^{-1}_{X_{L'}} &\text{if}\ \frac{3}{5}\leq t\leq \frac{4}{5},\\
		\phi^{(3h(t)-2)}_{X_{L''}}\circ\psi_{1}\circ\phi^{-1}_{X_{L'}}&\text{if}\ \frac{4}{5}\leq t\leq 1.
	\end{cases}
\end{align*}
Then $\rho_t$ is smooth in $t$ and it connects $\rho_0=\text{Id}$ with $\rho_1=\phi^{1}_{X_{L''}}\circ\psi_{1}\circ\phi^{-1}_{X_{L'}}=\phi_{L''}\circ\psi_1\circ\phi_{L'}^{-1}$. We now check that $\rho_t(L')\in\mathcal{U}'$ for all $t\in[0,1]$. So we have to check that $\rho_t(L')$ corresponds with a section in $\Gamma_{V'}(T^{*}L)$. As $L',L''\in\mathcal{U}'$, they correspond with sections $\alpha',\alpha''\in\Gamma_{V'}(T^{*}L)$. 
\begin{itemize}
	\item For $0\leq t\leq\frac{1}{5}$, we have that $\rho_t(L')$ corresponds with $(1-3h(t))\alpha'\in\Gamma_{V'}(T^{*}L)$.
	\item For $\frac{1}{5}\leq t\leq\frac{2}{5}$, we have that $\rho_t(L')$ corresponds with the zero section $0\in\Gamma_{V'}(T^{*}L)$.
	\item For $\frac{2}{5}\leq t\leq\frac{3}{5}$, we have that $0\leq 3h(t)-1\leq 1$. By construction, we have that $\psi_{s}(L)=L$ for all $0\leq s\leq 1$. It follows that $\rho_t(L')=\psi_{(3h(t)-1)}(L)=L$ corresponds with the zero section $0\in\Gamma_{V'}(T^{*}L)$.
	\item For $\frac{3}{5}\leq t\leq\frac{4}{5}$, we still have that $\rho_t(L')=\psi_1(L)=L$ corresponds with the zero section $0\in\Gamma_{V'}(T^{*}L)$.
	\item For $\frac{4}{5}\leq t\leq 1$, we have that $\rho_t(L)$ corresponds with $(3h(t)-2)\alpha''\in\Gamma_{V'}(T^{*}L)$.
\end{itemize}
In the first and last bullet point, we used that $V'\subset T^{*}L$ is fiberwise convex to conclude that the scalar multiples of $\alpha'$ and $\alpha''$ still take values in $V'$. The proof is now finished.
\end{proof}

\begin{remark}
A loose interpretation of the correspondence in Prop.~\ref{prop:bijection} is the following. Let $\text{Diff}(M)$ be the diffeomorphism group of $M$, which acts on the space of pairs $\mathcal{D}(\omega,L)$. The subgroup $\text{Diff}_{L}(M)$ of diffeomorphisms preserving $L$ acts on $\text{Def}_{L}(\omega)$. One can think of $\text{Def}_{L}(\omega)$ as a $\text{Diff}_{L}(M)$-slice to the $\text{Diff}(M)$-action on $\mathcal{D}(\omega,L)$. Indeed, we have a map
\[
\text{Diff}(M)\times_{\text{Diff}_{L}(M)}\text{Def}_{L}(\omega)\rightarrow\mathcal{D}(\omega,L):(f,\omega')\mapsto\big((f^{-1})^{*}\omega',f(L)\big)
\]
and upon restricting to $\mathcal{D}_{\mathcal{U}'}(\omega,L)$ we obtain a map in the other direction. From this point of view, one expects the local moduli space of simultaneous deformations to be the quotient $\text{Def}_{L}(\omega)/\text{Diff}_{L}(M)$ of the ``slice'' $\text{Def}_{L}(\omega)$. Prop.~\ref{prop:bijection} above makes this precise. 
\end{remark}

\section{The main result}
In this section, we will show that the local moduli space of simultaneous deformations can be identified with an open neighborhood of the origin in the relative cohomology $H^{2}(M,L)$, provided that we restrict to small enough deformations of the pair $(\omega,L)$. By invoking the bijection of moduli spaces established in \S\ref{sec:two}, we can essentially reduce the problem to proving the same result for the local moduli space of symplectic forms for which $L$ is Lagrangian.

	\vspace{0.2cm}
To make everything precise, we need to introduce suitable topologies on the deformation spaces	$\mathcal{D}_{\mathcal{U}'}(\omega,L)$ and 
$\text{Def}_{L}(\omega)$. In the rest of this section, we assume the following setup:

\begin{itemize}
	\item We endow the space $\text{Def}_{L}(\omega)$ with the $\mathcal{C}^{0}$-topology inherited from $\big(\Omega^{2}(M),\mathcal{C}^{0}\big)$.
	\item By identifying submanifolds in $\mathcal{U}'$ with their corresponding sections in $\Gamma_{V'}(T^{*}L)$, we can endow $\mathcal{U}'$ with the $\mathcal{C}^{1}$-topology. We then endow $\mathcal{D}_{\mathcal{U}'}(\omega,L)$  with the product $\big(\mathcal{C}^{0}\times\mathcal{C}^{1}\big)$-topology inherited from $
	\big(\Omega^{2}(M)\times \mathcal{U}',\mathcal{C}^{0}\times\mathcal{C}^{1}\big)
	$.
\end{itemize}
	
\subsection{Local parametrization of the moduli space $\text{Def}_{L}(\omega)/\sim$}

The moduli space of symplectic forms for which $L$ is Lagrangian can be identified with an open neighborhood of the origin in $H^{2}(M,L)$, provided that we restrict to small deformations of $\omega$ in $\text{Def}_{L}(\omega)$.

\begin{lemma}\label{lem:nbhd}
There is a $\mathcal{C}^{0}$-open $\mathcal{W}$ around $\omega$ in the space of closed two-forms $\Omega_{cl}^{2}(M)$ such that $\mathcal{W}$ is convex and $\mathcal{W}\subset\Omega_{symp}(M)$.
\end{lemma}
\begin{proof}
Denote by $\pi:=-\omega^{-1}$ the Poisson structure corresponding with $\omega$. Fix a Riemannian metric and consider the map $\|\cdot\|:\Omega^{2}_{cl}(M)\rightarrow\mathbb{R}$ determined by
\[
\|\tilde{\omega}\|:=\sup\left\{\left|\pi^{\sharp}\circ\tilde{\omega}^{\flat}(v)+v\right|: v\in TM, |v|=1\right\},\hspace{0.5cm}\tilde{\omega}\in\Omega^{2}_{cl}(M).
\]
Note that $\|\omega\|=0$ and that  $\|\tilde{\omega}\|\geq 1$ if $\tilde{\omega}$ is degenerate. Hence setting 
\[
\mathcal{W}:=\{\tilde{\omega}\in\Omega^{2}_{cl}(M):\|\tilde{\omega}\|<1\},
\]
we get a $\mathcal{C}^{0}$-open around $\omega$ in $\Omega_{cl}^{2}(M)$ consisting of symplectic forms. Note that $\mathcal{W}$ is convex. Indeed, take $\tilde{\omega}_{1},\tilde{\omega}_2\in\mathcal{W}$ and $t\in[0,1]$ and let $v\in TM$ be such that $|v|=1$. We then have
\begin{align*}
\left|\pi^{\sharp}\circ\left(t\tilde{\omega}_1^{\flat}+(1-t)\tilde{\omega}_2^{\flat}\right)(v)+v\right|&=\left|t\left(\pi^{\sharp}\circ\tilde{\omega}_1^{\flat}(v)+v\right)+(1-t)\left(\pi^{\sharp}\circ\tilde{\omega}_2^{\flat}(v)+v\right)\right|\\
&\leq t\left|\pi^{\sharp}\circ\tilde{\omega}_1^{\flat}(v)+v\right|+(1-t)\left|\pi^{\sharp}\circ\tilde{\omega}_2^{\flat}(v)+v\right|\\
&\leq t\|\tilde{\omega}_1\|+(1-t)\|\tilde{\omega}_2\|\\
&<t+(1-t)\\
&=1.
\end{align*}
This shows that $\|t\tilde{\omega}_1+(1-t)\tilde{\omega}_2\|<1$, hence $t\tilde{\omega}_1+(1-t)\tilde{\omega}_2\in\mathcal{W}$. So $\mathcal{W}$ is indeed convex.
\end{proof}

We now restrict to deformations of $\omega$ lying in the $\mathcal{C}^{0}$-open $\text{Def}_{L}(\omega)\cap\mathcal{W}$. The properties of the neighborhood $\mathcal{W}$ allow us to describe the moduli space  $\big(\text{Def}_{L}(\omega)\cap\mathcal{W}\big)/\sim$ explicitly.

\begin{prop}\label{prop:H2}
The map
\begin{equation}\label{eq:param}
\frac{\text{Def}_{L}(\omega)\cap\mathcal{W}}{\sim}\rightarrow H^{2}(M,L):[\omega']\mapsto[\omega-\omega']
\end{equation}
is a well-defined bijection onto an open neighborhood of the origin in $H^{2}(M,L)$.
\end{prop}
\begin{proof}
We first check that the map is well-defined. If $\omega'_{1},\omega'_{2}\in\text{Def}_{L}(\omega)\cap\mathcal{W}$ are equivalent then there exists an isotopy $(\rho_t)_{t\in[0,1]}$ of $M$ such that 
\[
(\rho_1)^{*}\omega'_{2}=\omega'_{1}\hspace{0.5cm}\text{and}\hspace{0.5cm}\rho_t(L)=L\ \  \text{for all}\ t\in[0,1].
\]
If $Y_t$ is the time-dependent vector field of $\rho_t$, then we have
\[
(\omega-\omega'_{2})-(\omega-\omega'_{1})=(\rho_1)^{*}\omega'_{2}-\omega'_{2}=\int_{0}^{1}\frac{d}{dt}(\rho_{t}^{*}\omega'_{2})dt=d\left(\int_{0}^{1}\rho_{t}^{*}(\iota_{Y_t}\omega'_{2})dt\right).
\]
Note that $\rho_{t}^{*}(\iota_{Y_t}\omega'_{2})\in\Omega^{1}(M,L)$ since $\rho_t$ preserves $L$, $Y_t$ is tangent to $L$ and $\omega'_{2}\in\Omega^{2}(M,L)$. Hence the classes $[\omega-\omega'_{1}]$ and $[\omega-\omega'_{2}]$ in $H^{2}(M,L)$ agree, so the map is well-defined.

Next, we check that the map surjects onto an open neighborhood of the origin in $H^{2}(M,L)$. Clearly, the class $[\omega]$ is mapped to zero. It remains to show that the image of the map \eqref{eq:param} is open. Assume that $[\beta]\in H^{2}(M,L)$ is the image of $[\omega']$ for some $\omega'\in\text{Def}_{L}(\omega)\cap\mathcal{W}$. Pick closed forms $\eta_1,\ldots,\eta_k\in\Omega^{2}(M,L)$ whose cohomology classes $[\eta_1],\ldots,[\eta_k]$ form a basis of $H^{2}(M,L)$. There exists $\epsilon>0$ such that for all $\delta\in\mathbb{R}^{k}$ with $\|\delta\|<\epsilon$, we have that
\[
\omega'-\sum_{i=1}^{k}\delta_i\eta_i\in\text{Def}_{L}(\omega)\cap\mathcal{W}.
\]
The image of its equivalence class in $\big(\text{Def}_{L}(\omega)\cap\mathcal{W}\big)/\sim$ under the map \eqref{eq:param} is
\[
\left[\omega-\omega'+\sum_{i=1}^{k}\delta_i\eta_i\right]=[\beta]+\sum_{i=1}^{k}\delta_i[\eta_i].
\]
Therefore the open ball $B_{[\beta],\epsilon}\subset H^{2}(M,L)$  is contained in the image of the map \eqref{eq:param}.

At last, we check that the map is injective. Assume that $\omega'_{1},\omega'_{2}\in\text{Def}_{L}(\omega)\cap\mathcal{W}$ are such that $[\omega'_1]=[\omega'_2]$ in $H^{2}(M,L)$. We have to show that there is an isotopy $(\rho_t)_{t\in[0,1]}$ such that 
\begin{equation}\label{eq:two}
(\rho_1)^{*}\omega'_{2}=\omega'_{1}\hspace{0.5cm}\text{and}\hspace{0.5cm}\rho_t(L)=L\ \  \text{for all}\ t\in[0,1].
\end{equation}
Pick $\beta\in\Omega^{1}(M,L)$ such that 
$$
\omega'_{2}-\omega'_1=d\beta
$$
and consider the straight line homotopy 
\[
\Omega_t=\omega'_1+td\beta.
\]
Towards using the Moser trick, let $(\rho_t)_{t\in[0,1]}$ be an isotopy with corresponding time-dependent vector field $(Y_t)_{t\in[0,1]}$. We then have
\[
\frac{d}{dt}(\rho_t)^{*}\Omega_t=(\rho_t)^{*}\left(d\left(\iota_{Y_t}\Omega_t+\beta\right)\right).
\]
Note that $\Omega_t$ is symplectic for all $t\in[0,1]$ since $\mathcal{W}$ is convex and $\mathcal{W}\subset\Omega_{symp}(M)$. Hence, we can pick $Y_t$ such that 
\begin{equation}\label{eq:perp}
	\iota_{Y_t}\Omega_t=-\beta.
\end{equation}
The isotopy $(\rho_t)_{t\in[0,1]}$ then satisfies 
\[
(\rho_1)^{*}\omega'_{2}=\omega'_1.
\]
It remains to argue that $\rho_t$ preserves $L$ for all $t\in[0,1]$. To do so, note that equation \eqref{eq:perp} implies that for all $x\in L$ we have
\[
(Y_t)(x)\in T_{x}L^{\perp_{\Omega_t}}=T_{x}L,
\]
where $T_{x}L^{\perp_{\Omega_t}}$ is the symplectic orthogonal of $T_{x}L$ with respect to $\Omega_t$. Here we used the fact that $L$ is Lagrangian with respect to $\Omega_t$ for all $t\in[0,1]$. So the time-dependent vector field $Y_t$ is tangent to $L$, and therefore the isotopy $\rho_t$ fixes $L$ for all $t\in[0,1]$. It follows that the conditions \eqref{eq:two} hold. This shows that the map \eqref{eq:param} is injective, so the proof is finished.
\end{proof}

\subsection{Continuity of the map $q$}
In the previous subsection, we parametrized a $\mathcal{C}^{0}$-open around $\omega$ in $\text{Def}_{L}(\omega)$ up to isotopies fixing $L$. We want to use this result to parametrize a $\big(\mathcal{C}^{0}\times\mathcal{C}^{1}\big)$-open around $(\omega,L)$ in $\mathcal{D}_{\mathcal{U}'}(\omega,L)$ up to isotopies. This can be done using the map
\[
q:\mathcal{D}_{\mathcal{U}'}(\omega,L)\rightarrow\text{Def}_{L}(\omega):(\omega',L')\mapsto\phi_{L'}^{*}\omega',
\]
if we show that this map is continuous with respect to the $\big(\mathcal{C}^{0}\times\mathcal{C}^{1}\big)$-topology on the domain and the $\mathcal{C}^{0}$-topology on the codomain. This subsection is devoted to proving the following.

\begin{prop}\label{prop:continuity}
The map
\[
F:\big(\Omega^{2}(M)\times \mathcal{U}',\mathcal{C}^{0}\times\mathcal{C}^{1}\big)\rightarrow \big(\Omega^{2}(M),\mathcal{C}^{0}\big):(\omega',L')\mapsto\phi_{L'}^{*}\omega'.
\]
is continuous at the point $(\omega,L)$.
\end{prop}

Towards proving Prop.\,\ref{prop:continuity}, we first note that the map
\[
\big(\Omega^{2}(M)\times\text{Diff}(M),\mathcal{C}^{0}\times\mathcal{C}^{1}\big)\rightarrow\big(\Omega^{2}(M),\mathcal{C}^{0}\big):(\omega',\phi)\mapsto\phi^{*}\omega'.
\]
is continuous. This follows immediately from the coordinate expression of the pullback $\phi^{*}\omega'$. Therefore, in order to prove Prop.\,\ref{prop:continuity} we only need to establish the following auxiliary result.

\begin{lemma}\label{lem:flow}
The assignment
\[
\big(\mathcal{U}',\mathcal{C}^{1}\big)\rightarrow\big(\text{Diff}(M),\mathcal{C}^{1}\big):L'\mapsto \phi_{L'}
\]
is continuous at the point $L$.
\end{lemma}

We will use the following version of Gronwall's inequality \cite[Thm.\,1.2.2]{Bellman}: if $u:[0,1]\rightarrow\mathbb{R}$ is a continuous, positive function and there are positive constants $A$ and $B$ such that
\begin{equation}\label{eq:gronwall}
u(t)\leq A+B\int_{0}^{t}u(s)ds
\end{equation}
then $u$ satisfies $u(t)\leq Ae^{B}$ for all $t\in[0,1]$. We now proceed with the proof of Lemma\,\ref{lem:flow}, which is inspired by the proofs of \cite[Lemma\,B.3]{Miranda} and \cite[Lemma\,3.10]{Marcut}.

\begin{proof}[Proof of Lemma\,\ref{lem:flow}]
The proof consists of two steps.

\begin{enumerate}
\item We first show that the map
\begin{equation}\label{eq:flow1}
\big(\mathcal{U}',\mathcal{C}^{1}\big)\rightarrow\big(\mathfrak{X}(M),\mathcal{C}^{1}\big):L'\mapsto X_{L'}
\end{equation}
is continuous at $L$. The support of $X_{L'}$ is contained in the open $U\subset M$, which we identify with $V\subset T^{*}L$. Hence, we may assume that $X_{L'}$ is a compactly supported vector field on $V\subset T^{*}L$. Recall that $L'\in\mathcal{U}'$ corresponds with a section $\alpha'\in\Gamma_{V'}(T^{*}L)$. In cotangent coordinates $(x_1,\ldots,x_n,y_1,\ldots,y_n)$, the assignment \eqref{eq:flow1} is given by
\[
\sum_{i=1}^{n}g_i dx_i\mapsto f\left(\sum_{i=1}^{n}g_i\partial_{y_i}\right),
\]
where $f$ is a fixed compactly supported function on $V$. From this expression, it is immediately clear that the assignment $L'\mapsto X_{L'}$ is continuous for $\mathcal{C}^{1}$-topologies. 
\item Next, we show that the map
\begin{equation}\label{eq:flow}
\big(\{X_{L'}:L'\in\mathcal{U}'\},\mathcal{C}^{1}\big)\rightarrow\big(\text{Diff}(M),\mathcal{C}^{1}\big):X_{L'}\mapsto\phi_{L'}
\end{equation}
is continuous at zero. Again, since $X_{L'}$ and its flow $\phi_{L'}$ are supported in the open $U\subset M$, we may assume that both are defined on the open $V\subset T^{*}L$. 

We introduce some notation. Cover $L$ in finitely many coordinate charts $O_i\cong\mathbb{R}^{n}$ such that the closed balls $\{\overline{B}_{i}\}_{i}$ of radius $1$ still cover $L$. Let us denote the cotangent coordinates on $O_i\times\mathbb{R}^{n}$ by $(z_1,\ldots,z_n,z_{n+1},\ldots,z_{2n})$. Since the compactly supported vector field $X_{L'}\in\mathfrak{X}(V)$ is vertical, it is locally given by
\[
X_{L'}|_{O_i\times\mathbb{R}^{n}}=f_i^{1}(z)\partial_{z_{n+1}}+\cdots+f_i^{n}(z)\partial_{z_{2n}}.
\]
Its $\mathcal{C}^{1}$-norm is then given by
\[
\|X_{L'}\|_{1}:=\sup_{i,j}\left\{\left|\frac{\partial^{|I|}f_{i}^{j}}{\partial z^{I}}(z)\right|: z\in\overline{B}_{i}\times\mathbb{R}^{n}\ \text{and}\ 0\leq|I|\leq 1 \right\}.
\] 
We denote the flow of $X_{L'}$ by $(\phi_t)_{t\in[0,1]}$, so that $\phi_{L'}=\phi_1$. 
In what follows, we need to estimate the $\mathcal{C}^{1}$-distance between $\phi_{1}$ and $\text{Id}$. To this end, we define
\[
\chi_{i,t}:=\phi_{t}-\text{Id}:O_i\times\mathbb{R}^{n}\rightarrow O_i\times\mathbb{R}^{n}.
\]
Written in components, we have $\chi_{i,t}=(0,\chi_{i,t}^{1},\ldots,\chi_{i,t}^{n})$.
Note that $\chi_{i,t}^{j}$ satisfies 
\[
\frac{d\chi_{i,t}^{j}}{dt}(z)=f_{i}^{j}(z+\chi_{i,t}(z)),\hspace{0.5cm}\chi_{i,0}^{j}=0.
\]
It follows that
\begin{equation}\label{eq:integrated}
\chi_{i,t}^{j}(z)=\int_{0}^{t}f_{i}^{j}(z+\chi_{i,s}(z))ds.
\end{equation}
We first estimate the difference between $\phi_{1}$ and $\text{Id}$. For $z\in\overline{B}_{i}\times\mathbb{R}^{n}$, we have
\begin{equation}\label{eq:est1}
\left|\chi_{i,1}^{j}(z)\right|=\left|\int_{0}^{1}f_{i}^{j}(z+\chi_{i,s}(z))ds\right|\leq \int_{0}^{1}\left|f_{i}^{j}(z+\chi_{i,s}(z))\right|ds\leq \|X_{L'}\|_{1}.
\end{equation}
We now estimate the difference between derivatives of $\phi_{1}$ and $\text{Id}$. By taking partial derivatives in the equation \eqref{eq:integrated}, we get for $z\in\overline{B}_{i}\times\mathbb{R}^{n}$ that
\[
\frac{\partial\chi_{i,t}^{j}}{\partial z_k}(z)=\int_{0}^{t}\left(\frac{\partial f_{i}^{j}}{\partial z_k}(z+\chi_{i,s}(z))+\sum_{l=n+1}^{2n}\frac{\partial f_{i}^{j}}{\partial z_l}(z+\chi_{i,s}(z))\frac{\partial\chi_{i,s}^{l-n}}{\partial z_k}(z)\right)ds.
\]
It follows that
\[
\left|\frac{\partial\chi_{i,t}^{j}}{\partial z_k}(z)\right|\leq \|X_{L'}\|_{1}+\|X_{L'}\|_{1}\int_{0}^{t}\sum_{l=1}^{n}\left|\frac{\partial\chi_{i,s}^{l}}{\partial z_k}(z)\right|ds.
\]
Consequently,
\[
\sum_{l=1}^{n}\left|\frac{\partial\chi_{i,t}^{l}}{\partial z_k}(z)\right|\leq n\|X_{L'}\|_{1}+n\|X_{L'}\|_{1}\int_{0}^{t}\sum_{l=1}^{n}\left|\frac{\partial\chi_{i,s}^{l}}{\partial z_k}(z)\right|ds.
\]
We are now in the setup of Gronwall's inequality, see \eqref{eq:gronwall} above.
It follows that
\[
\left|\frac{\partial\chi_{i,t}^{j}}{\partial z_k}(z)\right|\leq \sum_{l=1}^{n}\left|\frac{\partial\chi_{i,t}^{l}}{\partial z_k}(z)\right|\leq n\|X_{L'}\|_{1} e^{n\|X_{L'}\|_{1}}.
\]
In particular, we proved that
\begin{equation}\label{eq:est2}
	\left|\frac{\partial\chi_{i,1}^{j}}{\partial z_k}(z)\right|\leq n\|X_{L'}\|_{1} e^{n\|X_{L'}\|_{1}}.
\end{equation}
The estimates \eqref{eq:est1} and \eqref{eq:est2} confirm that the map \eqref{eq:flow} is indeed continuous. \qedhere
\end{enumerate}
\end{proof}

\subsection{Local parametrization of the moduli space $\mathcal{D}_{\mathcal{U}'}(\omega,L)/\sim$}\label{subsec:main}
We can now prove our main result, which parametrizes a $\big(\mathcal{C}^{0}\times\mathcal{C}^{1}\big)$-open around $(\omega,L)$ in $\mathcal{D}_{\mathcal{U}'}(\omega,L)$ up to isotopies. This open neighborhood of $(\omega,L)$ is obtained as follows:
\begin{enumerate}[i)]
	\item In Lemma\,\ref{lem:nbhd}, we constructed a suitable $\mathcal{C}^{0}$-open neighborhood  $\mathcal{W}$ of $\omega$ in $\Omega^{2}_{cl}(M)$. Hence there exists a $\mathcal{C}^{0}$-open $\mathcal{V}$ around $\omega$ in $\Omega^{2}(M)$ such that $\mathcal{W}=\mathcal{V}\cap \Omega^{2}_{cl}(M)$.
	\item By Prop.\,\ref{prop:continuity}, there exists a $\big(\mathcal{C}^{0}\times\mathcal{C}^{1}\big)$-open $\mathcal{O}$ around $(\omega,L)$ in $\Omega^{2}(M)\times \mathcal{U}'$ such that $F(\mathcal{O})\subset\mathcal{V}$. Then $\mathcal{D}_{\mathcal{U}'}(\omega,L)\cap\mathcal{O}$ is the desired $\big(\mathcal{C}^{0}\times\mathcal{C}^{1}\big)$-open around $(\omega,L)$ in $\mathcal{D}_{\mathcal{U}'}(\omega,L)$.
\end{enumerate}

\begin{thm}\label{thm:main}
The map
\begin{equation}\label{eq:main-map}
\frac{\mathcal{D}_{\mathcal{U}'}(\omega,L)\cap\mathcal{O}}{\sim}\rightarrow H^{2}(M,L):[(\omega',L')]\mapsto[\omega-\phi_{L'}^{*}\omega']
\end{equation}
is a well-defined bijection onto an open neighborhood of the origin in $H^{2}(M,L)$.
\end{thm}

\begin{proof}
The proof relies on the fact that if $(\omega',L')\in\mathcal{D}_{\mathcal{U}'}(\omega,L)\cap\mathcal{O}$, then $\phi_{L'}^{*}\omega'\in\text{Def}_{L}(\omega)\cap\mathcal{W}$. This follows from the construction of the neighborhood $\mathcal{O}$ outlined above.

To see that the map \eqref{eq:main-map} is well-defined, assume that $(\omega',L'),(\omega'',L'')\in\mathcal{D}_{\mathcal{U}'}(\omega,L)\cap\mathcal{O}$ are equivalent. First, since the map
\[
\overline{q}:\frac{\mathcal{D}_{\mathcal{U}'}(\omega,L)}{\sim}\rightarrow\frac{\text{Def}_{L}(\omega)}{\sim}:[(\eta,K)]\mapsto\big[\phi_{K}^{*}\eta\big]
\]
is well-defined by Lemma\,\ref{lem:well-defined}, it follows that $\phi_{L'}^{*}\omega',\phi_{L''}^{*}\omega''\in\text{Def}_{L}(\omega)\cap\mathcal{W}$ are equivalent. Next, since the map
\[
\frac{\text{Def}_{L}(\omega)\cap\mathcal{W}}{\sim}\rightarrow H^{2}(M,L):[\eta]\mapsto[\omega-\eta]
\]
is well-defined by Prop.\,\ref{prop:H2}, the classes $[\omega-\phi_{L'}^{*}\omega']$ and $[\omega-\phi_{L''}^{*}\omega'']$ in $H^{2}(M,L)$ agree. 

To see that the map \eqref{eq:main-map} is injective, assume that $(\omega',L'),(\omega'',L'')\in\mathcal{D}_{\mathcal{U}'}(\omega,L)\cap\mathcal{O}$ are such that $[\omega-\phi_{L'}^{*}\omega']=[\omega-\phi_{L''}^{*}\omega'']$ in $H^{2}(M,L)$. First, because of injectivity in Prop.\,\ref{prop:H2}, we get that $\phi_{L'}^{*}\omega',\phi_{L''}^{*}\omega''\in\text{Def}_{L}(\omega)\cap\mathcal{W}$ are equivalent. Next, because of injectivity in Prop.\,\ref{prop:bijection}, it follows that $(\omega',L'),(\omega'',L'')\in\mathcal{D}_{\mathcal{U}'}(\omega,L)\cap\mathcal{O}$ are equivalent. 

We now check that the map \eqref{eq:main-map} surjects onto an open neighborhood of zero in $H^{2}(M,L)$. Since the class $[(\omega,L)]$ is mapped to zero, we only need to show that the image of the map \eqref{eq:main-map} is open. We will argue as in the proof of Prop.\,\ref{prop:H2}. Assume that $[\beta]\in H^{2}(M,L)$ is the image of $[(\omega',L')]$ for some $(\omega',L')\in\mathcal{D}_{\mathcal{U}'}(\omega,L)\cap\mathcal{O}$. Pick closed forms $\eta_1,\ldots,\eta_k\in\Omega^{2}(M,L)$ whose cohomology classes $[\eta_1],\ldots,[\eta_k]$ form a basis of $H^{2}(M,L)$. Because the map
\[
\big(\Omega^{2}(M),\mathcal{C}^{0}\big)\rightarrow\big(\Omega^{2}(M),\mathcal{C}^{0}\big):\alpha\mapsto (\phi_{L'}^{-1})^{*}\alpha
\]
is continuous, there exists $\epsilon>0$ such that for all $\delta\in\mathbb{R}^{k}$ with $\|\delta\|<\epsilon$, we have that
\[
\left(\omega'-(\phi_{L'}^{-1})^{*}\Big(\sum_{i=1}^{k}\delta_i\eta_i\Big),L'\right)\in\mathcal{D}_{\mathcal{U}'}(\omega,L)\cap\mathcal{O}.
\]
The image of its equivalence class in $\big(\mathcal{D}_{\mathcal{U}'}(\omega,L)\cap\mathcal{O}\big)/\sim$ under the map \eqref{eq:main-map} is 
\[
\left[\omega-\phi_{L'}^{*}\omega'+\sum_{i=1}^{k}\delta_i\eta_i\right]=[\beta]+\sum_{i=1}^{k}\delta_i[\eta_i].
\]
Therefore the open ball $B_{[\beta],\epsilon}\subset H^{2}(M,L)$  is contained in the image of the map \eqref{eq:main-map}. This shows that the image of the map \eqref{eq:main-map} is an open neighborhood of the origin in $H^{2}(M,L)$.
\end{proof}

	\section{Relation with other deformation problems}\label{sec:relation}
	The problem studied in this note concerns simultaneous deformations of a symplectic form $\omega\in\Omega^{2}(M)$ and a Lagrangian submanifold $L\subset(M,\omega)$. One could also deform both structures individually, i.e. study Lagrangian deformations of $L$ for fixed $\omega$ or study symplectic deformations of $\omega$. The aim of this section is to relate these deformation problems.

	\subsection{Deformations of Lagrangian submanifolds}\label{subsec:lagr}
	Let $(M,\omega)$ be a symplectic manifold and $L\subset(M,\omega)$ a compact Lagrangian submanifold. We recall the well-known classification of Lagrangian deformations $L'$ of $L$, up to equivalence by Hamiltonian isotopies.
	
	As in \S\,\ref{sec:one}, we will restrict to Lagrangian submanifolds $L'\subset(M,\omega)$ lying in a chart for the non-linear Grassmannian $\text{Gr}_{L}(M)$ around $L$. By Weinstein's Lagrangian neighborhood theorem \cite{Weinstein}, we can fix a symplectomorphism between a neighborhood $V$ of $L$ in $(T^{*}L,\omega_{can})$ and a neighborhood $U$ of $L$ in $(M,\omega)$, which restricts to the identity on $L$. We denote it by
	\[
	\psi:(V,\omega_{can})\rightarrow (U,\omega).
	\]
	Let us also fix smaller neighborhoods $V'$ and $U'$ of $L$ which correspond under $\psi$ and satisfy $L\subset V'\subset\overline{V'}\subset V$ and $L\subset U'\subset\overline{U'}\subset U$. We also make sure that $V'$ is fiberwise convex. We define
	\[
	\Gamma_{V'}(T^{*}L)=\{\sigma\in\Gamma(T^{*}L):\ \sigma(L)\subset V'\}
	\]
	and set $\mathcal{U}'\subset \text{Gr}_{L}(M)$ to be the neighborhood of $L$ consisting of submanifolds that are images of $\psi\circ\sigma:L\rightarrow M$ for $\sigma\in\Gamma_{V'}(T^{*}L)$. We restrict to Lagrangian submanifolds in $\mathcal{U}'$.
	
	\begin{defi}
		\begin{enumerate}[i)]
			\item We introduce the space
			\[
			\mathcal{D}_{\mathcal{U}'}(L):=\left\{L'\in\mathcal{U}': L'\subset(M,\omega)\ \text{is Lagrangian}\right\}.
			\]	
			\item For $L',L''\in\mathcal{D}_{\mathcal{U}'}(L)$, we say that $L'\sim L''$ if there exists a Hamiltonian isotopy $(\phi_t)_{t\in[0,1]}$ of $(M,\omega)$ such that 
			\[
			\phi_1(L')=L''\hspace{0.5cm}\text{and}\hspace{0.5cm}\phi_t(L')\in\mathcal{U}'\ \text{for all}\ t\in[0,1].
			\]
		\end{enumerate}
	\end{defi}
	
	A Lagrangian submanifold $L'\in\mathcal{D}_{\mathcal{U}'}(L)$ corresponds to a unique section $\alpha'\in\Gamma_{V'}(T^{*}L)$ whose graph is Lagrangian with respect to $\omega_{can}$. It is well-known that the latter condition is equivalent with $\alpha'\in\Omega^{1}(L)$ being closed \cite[Prop.\,3.4.2]{McDuff}. So we get an assignment
	\[
	\mathcal{D}_{\mathcal{U}'}(L)\rightarrow H^{1}(L):L'\mapsto[\alpha'].
	\]
	It was proved in \cite[Chapter~5, Prop.\,2.7]{Schaetz} that this map descends to a bijection between the moduli space $\mathcal{D}_{\mathcal{U}'}(L)/\sim$ and an open neighborhood of the origin in $H^{1}(L)$.

	\subsection{Deformations of symplectic forms}\label{sub:sympl}
	Let $M$ be compact and $\omega\in\Omega_{symp}(M)$ a symplectic form. We recall the well-known classification of deformations $\omega'\in\Omega_{symp}(M)$ of $\omega$, quotienting by the equivalence relation
	\begin{equation}
		\omega'\sim \omega''\Leftrightarrow \exists\phi\in\text{Diff}_{0}(M)\ \text{such that}\ \phi^{*}\omega''=\omega'.
	\end{equation}
	We restrict to a convex $\mathcal{C}^{0}$-open $\mathcal{D}(\omega)\subset\Omega^{2}_{cl}(M)$ around $\omega$ such that $\mathcal{D}(\omega)\subset\Omega_{symp}(M)$, see Lemma\,\ref{lem:nbhd}. It follows from Moser's theorem \cite[Thm.\,3.2.4]{McDuff} that the map
	\[
	\mathcal{D}(\omega)\rightarrow H^{2}(M):\omega'\mapsto[\omega-\omega']
	\]
	induces a bijection between $\mathcal{D}(\omega)/\sim$ and an open neighborhood of the origin in $H^{2}(M)$.

	

	\subsection{Relating the moduli spaces}
	Let $(M,\omega)$ be a compact symplectic manifold with a compact Lagrangian submanifold $L\subset(M,\omega)$. 
	\begin{itemize}
		\item In \S\ref{subsec:lagr}, we described $\mathcal{C}^1$-small deformations $L'$ of $L$ up to Hamiltonian isotopy. We denote the resulting moduli space by $\mathcal{M}_{L}$. It is parametrized by
		\begin{equation}\label{eq:Id1}
			\mathcal{M}_{L}\rightarrow H^1(L):[L']\mapsto[\alpha'].
		\end{equation}
		\item In \S\ref{subsec:main}, we described $(\mathcal{C}^{0}\times\mathcal{C}^{1})$-small deformations $(\omega',L')$ of the pair $(\omega,L)$ up to isotopies. We denote the resulting moduli space by $\mathcal{M}_{(\omega,L)}$. It is parametrized by
		\begin{equation}\label{eq:Id2}
			\mathcal{M}_{(\omega,L)}\rightarrow H^{2}(M,L):[(\omega',L')]\mapsto[\omega-\phi_{L'}^{*}\omega'].
		\end{equation}
		\item In \S\ref{sub:sympl}, we described $\mathcal{C}^{0}$-small deformations $\omega'$ of $\omega$ up to isotopies. We denote the resulting moduli space by $\mathcal{M}_{\omega}$. It is parametrized by
		\begin{equation}\label{eq:Id3}
			\mathcal{M}_{\omega}\rightarrow H^{2}(M):[\omega']\mapsto[\omega-\omega'].
		\end{equation}
	\end{itemize}
	These moduli spaces fit in a natural sequence
	\begin{equation}\label{eq:geom}
		\begin{tikzcd}[column sep=60pt]
			\mathcal{M}_{L}\arrow{r}{[L']\mapsto[(\omega,L')]}&\mathcal{M}_{(\omega,L)}\arrow{r}{[(\omega',L')]\mapsto [\omega']}&\mathcal{M}_{\omega}.
		\end{tikzcd}
	\end{equation}
	On the other hand, the cohomology groups $H^1(L),H^{2}(M,L)$ and $H^{2}(M)$ that model these moduli spaces fit in an exact sequence
	\begin{equation}\label{eq:alg}
		\begin{tikzcd}[column sep=60pt]
			H^{1}(L)\arrow{r}{[\alpha]\mapsto[d\widetilde{\alpha}]}&H^{2}(M,L)\arrow{r}{[\beta]\mapsto [\beta]}&H^{2}(M).
		\end{tikzcd}
	\end{equation}
	Here $\widetilde{\alpha}\in\Omega^{1}(M)$ is any one-form such that $\iota_{L}^{*}\widetilde{\alpha}=\alpha$. This sequence is part of the long exact sequence already mentioned in \eqref{eq:long-ex}. We claim that the ``geometric'' sequence \eqref{eq:geom} and the ``algebraic'' sequence \eqref{eq:alg} agree if we adopt the parametrizations in \eqref{eq:Id1}, \eqref{eq:Id2} and \eqref{eq:Id3}.
	
	\begin{prop}\label{prop:diag}
		The following diagram commutes
		\[
		\begin{tikzcd}[column sep=60pt, row sep=large]
			&\mathcal{M}_{L}\arrow{r}{[L']\mapsto[(\omega,L')]}\arrow{d}{[L']\mapsto[\alpha']}&\mathcal{M}_{(\omega,L)}\arrow{r}{[(\omega',L')]\mapsto[\omega']}\arrow{d}{[(\omega',L')]\mapsto\left[\omega-\phi_{L'}^{*}\omega'\right]}&\mathcal{M}_{\omega}\arrow{d}{[\omega']\mapsto[\omega-\omega']}.\\
			&H^{1}(L)\arrow{r}{[\alpha]\mapsto[d\widetilde{\alpha}]} &H^{2}(M,L)\arrow{r}{[\beta]\mapsto [\beta]} &H^{2}(M)
		\end{tikzcd}.
		\]
	\end{prop}
	\begin{proof}
		To see that the second square commutes, we just have to remark that $[\phi_{L'}^{*}\omega']=[\omega']$ in $H^{2}(M)$ because $\phi_{L'}\in\text{Diff}_{0}(M)$ induces the identity in cohomology. To show that the first square commutes, let us denote by $(\phi_t)_{t\in[0,1]}$ the flow of the vector field
		\[
		X_{L'}=f\psi_{*}(X_{\alpha'}),
		\]
		so that $\phi_{L'}=\phi_1$. We then have
		\begin{align*}
			\omega-\phi_{L'}^{*}\omega&=-\int_{0}^{1}\left(\frac{d}{dt}\phi_{t}^{*}\omega\right)dt=-d\left(\int_{0}^{1}\phi_t^{*}\big(\iota_{X_{L'}}\omega\big)dt\right),
		\end{align*}
		so the first square commutes as soon as we show that
		\begin{equation}\label{eq:to-show}
			\iota_{L}^{*}\left(\int_{0}^{1}\phi_t^{*}\big(\iota_{X_{L'}}\omega\big)dt\right)=-\alpha'.
		\end{equation}
		To prove this equality, take $x\in L$ and $v\in T_{x}L$ and compute
		\begin{align}\label{eq:simplify}
			&\int_{0}^{1}\omega_{\phi_t(x)}\big(X_{L'}(\phi_t(x)),(d_{x}\phi_t)(v)\big)dt\nonumber\\
			&\hspace{1.5cm}=\int_{0}^{1}f(\phi_t(x))\omega_{\phi_t(x)}\big((\psi_{*}X_{\alpha'})(\phi_t(x)),(d_{x}\phi_t)(v)\big)dt\nonumber\\
			&\hspace{1.5cm}=\int_{0}^{1}\omega_{\phi_t(x)}\big((\psi_{*}X_{\alpha'})(\phi_t(x)),(d_{x}\phi_t)(v)\big)dt\nonumber\\
			&\hspace{1.5cm}=\int_{0}^{1}\big(\omega_{can}\big)_{\psi^{-1}(\phi_t(x))}\big(X_{\alpha'}(\psi^{-1}(\phi_t(x))),(d_{\phi_t(x)}\psi^{-1})(d_{x}\phi_t(v))\big)dt,
		\end{align}
		where we used that $f\equiv 1$ along the flow line $(\phi_t(x))_{t\in[0,1]}$ and that $\psi^{*}\omega=\omega_{can}$. We now simplify the formula \eqref{eq:simplify}. Recall from Rem.\,\ref{rem:corr} that, if $p:T^{*}L\rightarrow L$ is the projection, then
		\begin{equation*}
			\iota_{X_{\alpha'}}\omega_{can}=-p^{*}\alpha'.
		\end{equation*}
		Inserting this equality into \eqref{eq:simplify}, we obtain
		\begin{equation}\label{eq:int}
			-\int_{0}^{1}(p^{*}\alpha')_{\psi^{-1}(\phi_t(x))}\left(d_{x}(\psi^{-1}\circ\phi_t)(v)\right)dt.
		\end{equation}
		By functoriality, $(\psi^{-1}\circ\phi_t\circ\psi)_{t\in[0,1]}$ is the flow of the vector field $\psi^{*}(f)X_{\alpha'}$ which is vertical. So we have that $p\circ \psi^{-1}\circ\phi_t=p\circ\psi^{-1}$, hence the expression \eqref{eq:int} reduces to 
		\[
		-\int_{0}^{1}\alpha'_{p(\psi^{-1}(x))}\left(d_{x}(p\circ\psi^{-1})(v)\right)dt=-\int_{0}^{1}\alpha'_{x}(v)dt=-\alpha'_{x}(v).
		\]
		Here we used that $\psi|_{L}=\text{Id}$. So the equality \eqref{eq:to-show} holds, hence the proof is finished.
	\end{proof}

	\section{Relation with formality of Koszul brackets}
	In this section, we remark that Thm.\,\ref{thm:main} is consistent with a formality result for Koszul brackets due to Fiorenza-Manetti \cite{Fiorenza}, in the following sense. Recall from Prop.\,\ref{prop:bijection} that we have a bijection of moduli spaces
	\[
	\overline{q}:\frac{\mathcal{D}_{\mathcal{U}'}(\omega,L)}{\sim}\rightarrow\frac{\text{Def}_{L}(\omega)}{\sim}:[(\omega',L')]\mapsto\big[\phi_{L'}^{*}\omega'\big].
	\]
	We will point out that the moduli space $\text{Def}_{L}(\omega)/\sim$ is in bijection with the moduli space of Maurer-Cartan elements of a dgL[1]a-structure on $\Omega^{\bullet}(M,L)[2]$. This dgL[1]a was shown to be homotopy abelian in \cite{Fiorenza}, indicating once more that the moduli space of deformations of the pair $(\omega,L)$ should be given by $H^{2}(M,L)$. We now spell this out in more detail.

\vspace{0.2cm}

	The most straightforward way of parametrizing small deformations of $\omega$ in $\text{Def}_{L}(\omega)$ would be through small closed elements of $\Omega^{2}(M,L)$. This leads to the local description of the moduli space $\text{Def}_{L}(\omega)/\sim$ established in Prop.\,\ref{prop:H2}.  In this section, we will look at a different approach. It turns out that the space $\text{Def}_{L}(\omega)$ can be parametrized alternatively via Maurer-Cartan elements of a dgL[1]a-structure on $\Omega^{\bullet}(M,L)[2]$ inherited from the Koszul dgL[1]a.

	\begin{defi}
		Let $(M,\omega)$ be a symplectic manifold with Poisson structure $\pi:=-\omega^{-1}$. We get an $L_{\infty}[1]$-algebra $\big(\Omega^{\bullet}(M)[2],\lambda_1,\lambda_2\big)$ whose multibrackets are defined as follows:
		\[
		\begin{cases}
			\lambda_1(\alpha)=d\alpha\\
			\lambda_2(\alpha,\beta)=(-1)^{|\alpha|}[\alpha,\beta]_{\pi}
		\end{cases}.
		\]
		Here $[-,-]_{\pi}$ is the Koszul bracket associated with $\pi$, see for instance \cite{Fiorenza},\cite{Koszul}. 
	\end{defi}

	We refer to $\big(\Omega^{\bullet}(M)[2],\lambda_1,\lambda_2\big)$ as the Koszul dgL[1]a associated with $\pi$. It was shown in \cite[Lemma\,2.12]{Koszul} that deformations of the symplectic form $\omega$ can be parametrized by small Maurer-Cartan elements of $\big(\Omega^{\bullet}(M)[2],\lambda_1,\lambda_2\big)$. In more detail, let $I_{\pi}$ be the neighborhood of $M\subset\wedge^{2}T^{*}M$ consisting of those bilinear forms $\beta$ such that 
	\[
	\text{Id}+\pi^{\sharp}\circ\beta^{\flat}:TM\rightarrow TM
	\]
	is invertible. For $\beta\in I_{\pi}$, we can define a two-form $F(\beta)$ determined by
	\begin{equation}\label{eq:F}
		F(\beta)^{\flat}=\beta^{\flat}\circ(\text{Id}+\pi^{\sharp}\circ\beta^{\flat})^{-1}.
	\end{equation}

	\begin{lemma}\label{lem:parametrization}
		There is a bijection between:
		\begin{enumerate}[i)]
			\item Small Maurer-Cartan elements $\beta$ of the Koszul dgL[1]a $\big(\Omega^{\bullet}(M)[2],\lambda_1,\lambda_2\big)$, i.e. 
			\[
			\beta\in I_{\pi}\ \ \text{such that}\ \ d\beta+\frac{1}{2}[\beta,\beta]_{\pi}=0,
			\]
			\item Deformations of the symplectic form $\omega$.
		\end{enumerate}
		The bijection assigns to $\beta\in I_{\pi}\cap MC\big(\Omega^{\bullet}(M)[2],\lambda_1,\lambda_2\big)$ the symplectic form $\omega+F(\beta)$.
	\end{lemma}

	The subspace $\Omega^{\bullet}(M,L)[2]\subset\Omega^{\bullet}(M)[2]$ inherits a dgL[1]a-structure $\big(\Omega^{\bullet}(M,L)[2],\lambda_1,\lambda_2\big)$, as shown in \cite[\S\,5]{Fiorenza}. This dgL[1]a yields an alternative parametrization for $\text{Def}_{L}(\omega)$, see \cite{ours}.

	\begin{lemma}\label{lem:mc}
		We get a bijection
		\[
		\left\{\beta\in I_{\pi}\cap\Omega^{2}(M,L)\ \ \text{such that}\ \ d\beta+\frac{1}{2}[\beta,\beta]_{\pi}=0\right\}\rightarrow\text{Def}_{L}(\omega):\beta\mapsto\omega+F(\beta).
		\]
	\end{lemma}

	The dgL[1]a $\big(\Omega^{\bullet}(M,L)[2],\lambda_1,\lambda_2\big)$ carries an algebraic notion of equivalence defined on its Maurer-Cartan set. Two Maurer-Cartan elements $\beta_0,\beta_1$ of $\big(\Omega^{\bullet}(M,L)[2],\lambda_1,\lambda_2\big)$ are called gauge equivalent if there is a one-parameter family of Maurer-Cartan elements $(\beta_t)_{t\in[0,1]}$ interpolating between $\beta_0$ and $\beta_1$, as well as a one-parameter family $(\alpha_t)_{t\in[0,1]}$ in $\Omega^{1}(M,L)$, such that 
	\[
	\frac{\partial}{\partial t}\beta_t=d\alpha_t-[\alpha_t,\beta_t]_{\pi}.
	\]
	Under the parametrization of $\text{Def}_{L}(\omega)$ by Maurer-Cartan elements of $\big(\Omega^{\bullet}(M,L)[2],\lambda_1,\lambda_2\big)$, the equivalence of elements in $\text{Def}_{L}(\omega)$ by isotopies fixing $L$ essentially agrees with the gauge equivalence of Maurer-Cartan elements. We again refer to \cite{ours} for a proof.

	\begin{prop}\label{prop:gauge}
		Let $\beta_0$ and $\beta_1$ be Maurer-Cartan elements of $\big(\Omega^{\bullet}(M,L)[2],\lambda_1,\lambda_2\big)$ lying in $I_{\pi}$, corresponding with $\omega+F(\beta_0)$ and $\omega+F(\beta_1)$ in $\text{Def}_{L}(\omega)$. The following are equivalent:
		\begin{enumerate}
			\item The Maurer-Cartan elements $\beta_0$ and $\beta_1$ are gauge equivalent through a family of Maurer-Cartan elements $(\beta_t)_{t\in[0,1]}$ lying in $I_{\pi}$.
			\item There is an isotopy $(\rho_t)_{t\in[0,1]}$ such that $\rho_{1}^{*}(\omega+F(\beta_1))=\omega+F(\beta_0)$ and $\rho_t(L)=L$ for all $t\in[0,1]$.
		\end{enumerate}
	\end{prop}
		
	Combining Prop.\,\ref{prop:bijection} with Lemma\,\ref{lem:mc} and Prop.\,\ref{prop:gauge}, we get a correspondence
	\begin{equation}\label{eq:correspondence}
		\frac{\mathcal{D}_{\mathcal{U}'}(\omega,L)}{\sim}\longrightarrow\frac{I_{\pi}\cap MC\big(\Omega^{\bullet}(M,L)[2],\lambda_1,\lambda_2\big)}{\sim_{gauge}}:[(\omega',L')]\mapsto\big[F^{-1}\big(\phi_{L'}^{*}\omega'-\omega\big)\big].
	\end{equation}

	\begin{remark}
		The correspondence  \eqref{eq:correspondence} is natural if one takes the point of view of Voronov's derived bracket construction \cite{Voronov}, which we now turn to describe. Assume that we are given a V-data, that is a quadruple $(\mathcal{L},\mathfrak{a},P,\Delta)$ where:
		\begin{itemize}
			\item $(\mathcal{L},[-,-])$ is a graded Lie algebra,
			\item $\mathfrak{a}$ is an abelian Lie subalgebra,
			\item $P:\mathcal{L}\rightarrow\mathfrak{a}$ is a linear projection whose kernel is a Lie subalgebra of $\mathcal{L}$,
			\item $\Delta\in\ker(P)$ is an element of degree $1$ such that $[\Delta,\Delta]=0$.
		\end{itemize}
		The derived bracket construction produces out of a V-data an $L_{\infty}[1]$-algebra structure on $\mathcal{L}[1]\oplus  \mathfrak{a}$ whose multibrackets are defined as follows. Below, we denote $D:=[\Delta,\bullet]:\mathcal{L}\rightarrow \mathcal{L}$.
		The differential is given by
		\[
		d(x[1],a)=\big(-D(x)[1],P(x+D(a))\big),
		\]
		the binary bracket satisfies
		\[
		\{x[1],y[1]\}=(-1)^{|x|}[x,y][1],
		\]
		and for $n\geq 1$ we have
		\begin{align*}
			\{x[1],a_1,\ldots,a_n\}&=P[\ldots[x,a_1],\ldots,a_n],\\
			\{a_1,\ldots,a_n\}&=P[\ldots[D(a_1),a_2],\ldots,a_n].
		\end{align*}
		Here $x,y\in \mathcal{L}$ and $a_1,\ldots,a_n\in\mathfrak{a}$. Up to permutation of the entries, all remaining multibrackets vanish.
		Restricting the multibrackets of $\mathcal{L}[1]\oplus  \mathfrak{a}$ to $\ker(P)[1]$ one obtains a dgL[1]a, and the inclusion 	$\ker(P)[1]\hookrightarrow \mathcal{L}[1]\oplus  \mathfrak{a}$ is a strict quasi-isomorphism of  $L_{\infty}[1]$-algebras. See the appendix of \cite{ours} for a proof of this fact, which is also hinted at in \cite[Section 4]{Voronov}.

		Now let $(M,\omega)$ be a symplectic manifold and $L\subset(M,\omega)$ a Lagrangian submanifold. Via the Lagrangian neighborhood theorem, we identify a neighborhood of $L$ in $M$ with a vector bundle $E:=T^{*}L$, where the latter is endowed with its canonical Poisson structure $\pi_{can}:=-\omega_{can}^{-1}$.
		We then obtain a V-data $(\mathcal{L},\mathfrak{a},P,\Delta)$ given by the following \cite[Lemma\,2.2]{Fregier}:  
		\begin{itemize}
			\item the graded Lie algebra $\mathcal{L}$ is the space of multivector fields $\mathfrak{X}^{\bullet}(E)[1]$ endowed with the Schouten-Nijenhuis bracket,
			\item the abelian Lie subalgebra $\mathfrak{a}$ is $\Gamma(\wedge^{\bullet}E)[1]$, which can be viewed as the space of vertical fiberwise constant multivector fields.
			\item the projection $P:\mathcal{L}\rightarrow\mathfrak{a}$ is obtained by restricting multivector fields to $L$ and then applying the projection $\wedge^{\bullet} TE|_{L}\rightarrow \wedge^{\bullet}E$ coming from the splitting $TE|_L=TL\oplus E$.
			\item the Maurer-Cartan element $\Delta\in\ker(P)$ is the Poisson structure $\pi_{can}$.
		\end{itemize}
		The derived bracket construction yields an $L_{\infty}[1]$-algebra structure on $\mathfrak{X}^{\bullet}(E)[2]\oplus\Gamma(\wedge^{\bullet}E)[1]$ which by \cite[Cor.\,2.5]{Fregier} governs the deformation problem of $(\omega_{can},L)$, at least heuristically.\footnote{The precise statement is the following. Rather than allowing all multivector fields on $E$, one should restrict to the ones that are fiberwise entire \cite{entire}. The $L_{\infty}[1]$-algebra structure on $\mathfrak{X}_{fe}^{\bullet}(E)[2]\oplus\Gamma(\wedge^{\bullet}E)[1]$ has the property that its Maurer-Cartan elements are pairs $(\pi,\phi)\in\mathfrak{X}_{fe}^{2}(E)\oplus\Gamma(E)$ such that $\pi_{can}+\pi$ is Poisson and $\text{Graph}(-\phi)$ is coisotropic with respect to $\pi_{can}+\pi$. Note that $\pi_{can}+\pi$ is the inverse of a symplectic form whenever $\pi$ is sufficiently small. Hence, small Maurer-Cartan elements of $\mathfrak{X}_{fe}^{\bullet}(E)[2]\oplus\Gamma(\wedge^{\bullet}E)[1]$ parametrize deformations of $(\omega_{can},L)$ for which the associated Poisson structure remains fiberwise entire.} 
		Restricting the multibrackets to $\ker(P)[1]$, we obtain a dgL[1]a with multibrackets
		\[
		\begin{cases}
			\mu_1(Q)=-[\pi_{can},Q]\\
			\mu_2(Q_1,Q_2)=-(-1)^{|Q_1|}[Q_1,Q_2]
		\end{cases},
		\]
		where degrees are taken in $\mathfrak{X}^{\bullet}(E)$. The dgL[1]a $\big(\ker(P)[1],\mu_1,\mu_2\big)$ is strictly isomorphic to the dgL[1]a $\big(\Omega^{\bullet}(E,L)[2],\lambda_1,\lambda_2\big)$ introduced above, via the map
		\[
		-\wedge^{\bullet}\pi_{can}^{\sharp}:\big(\Omega^{\bullet}(E,L)[2],\lambda_1,\lambda_2\big)\overset{\sim}{\longrightarrow}\big(\ker(P)[1],\mu_1,\mu_2\big).
		\]
		Altogether, we get a strict $L_{\infty}[1]$-quasi-isomorphism
		\[
		\big(\Omega^{\bullet}(E,L)[2],\lambda_1,\lambda_2\big)\longrightarrow \mathfrak{X}^{\bullet}(E)[2]\oplus\Gamma(\wedge^{\bullet}E)[1]:\beta\mapsto \big(-\wedge^{\bullet}\pi_{can}^{\sharp}(\beta),0\big).
		\]
		Heuristically, one expects this map to induce a bijection of equivalence classes of Maurer-Cartan elements. This is realized rigorously by the inverse of the correspondence in \eqref{eq:correspondence}. Indeed, the latter is given by
		\[
		\frac{I_{\pi}\cap MC\big(\Omega^{\bullet}(M,L)[2],\lambda_1,\lambda_2\big)}{\sim_{gauge}}\longrightarrow\frac{\mathcal{D}_{\mathcal{U}'}(\omega,L)}{\sim}:[\beta]\mapsto[(\omega+F(\beta),L)],
		\]
		and the Poisson structure corresponding with $\omega+F(\beta)$ is indeed $\pi-\wedge^{2}\pi^{\sharp}(\beta)$, see \cite[\S\,2]{Koszul}.
	\end{remark}
	

	We can relate the moduli space of Maurer-Cartan elements appearing in \eqref{eq:correspondence} with the cohomology group $H^{2}(M,L)$ thanks to a result by Fiorenza-Manetti \cite[\S\,5]{Fiorenza} stating that the dgL[1]a $\big(\Omega^{\bullet}(M,L)[2],\lambda_1,\lambda_2\big)$ is homotopy abelian. In more detail, let $R_{\pi}$ be the coderivation of $S\big(\Omega^{\bullet}(M)[2]\big)$ extending the second Koszul bracket
	\[
	\mathcal{K}(\iota_{\pi})_{2}:S^{2}\big(\Omega^{\bullet}(M)[2]\big)\rightarrow\Omega^{\bullet}(M)[2]:\alpha\odot\beta\mapsto\iota_{\pi}(\alpha\wedge\beta)-\iota_{\pi}\alpha\wedge\beta-\alpha\wedge\iota_{\pi}\beta.
	\]
	Then we have an isomorphism of $L_{\infty}[1]$-algebras
	\[
	e^{R_{\pi}}:\big(\Omega^{\bullet}(M)[2],\lambda_1,\lambda_2\big)\rightarrow\big(\Omega^{\bullet}(M)[2],d\big)
	\]
	which induces an isomorphism between $L_{\infty}[1]$-subalgebras
	\[
	e^{R_{\pi}}:\big(\Omega^{\bullet}(M,L)[2],\lambda_1,\lambda_2\big)\rightarrow\big(\Omega^{\bullet}(M,L)[2],d\big).
	\]
	The induced map between the Maurer-Cartan sets is well-defined for small Maurer-Cartan elements and coincides with the map $F$ defined in \eqref{eq:F}, see \cite[Prop.\,4.10]{Koszul}. Heuristically, this map should give a bijection between moduli spaces of Maurer-Cartan elements. Hence, combined with the correspondence \eqref{eq:correspondence}, one should get a parametrization
	\[
	\frac{\mathcal{D}_{\mathcal{U}'}(\omega,L)}{\sim}\longrightarrow H^{2}(M,L):[(\omega',L')]\mapsto\big[\phi_{L'}^{*}\omega'-\omega\big].
	\]
	Thm.\,\ref{thm:main} makes sense of this in a more precise way -- upon restricting to small deformations $\mathcal{D}_{\mathcal{U}'}(\omega,L)\cap\mathcal{O}\subset\mathcal{D}_{\mathcal{U}'}(\omega,L)$, we get a bijection onto an open around the origin in $H^{2}(M,L)$.

	\begin{remark}
		There is of course no need for Koszul brackets if one wants to describe small deformations of $\omega$ in $\text{Def}_{L}(\omega)$, as those are simply parametrized by small closed elements of $\Omega^{2}(M,L)$. Koszul brackets are useful however in more complicated deformation problems, for instance to describe deformations of a presymplectic form \cite{Koszul} or deformations of a symplectic form for which a given coisotropic submanifold remains coisotropic \cite{ours}. 
	\end{remark}

\end{document}